\documentclass[preprint,12pt]{elsarticle}

\usepackage{amssymb,amsmath}
\usepackage{amsthm,indentfirst}

\newcommand{\emm}{\mathrm{1\hspace{-0.3em} I}}

\newtheorem{thm}{Theorem}[section]
\newtheorem{lem}[thm]{Lemma}
\newtheorem{cor}[thm]{Corollary}
\newtheorem{rem}[thm]{Remark}
\newtheorem{prop}[thm]{Proposition}
\newtheorem{dfn}[thm]{Definition}

\journal{Advances in Mathematics}

\begin{document} 

\begin{frontmatter}

\title{Banach limits and traces on $\mathcal L_{1,\infty}$.}

\author[label1]{Evgeniy Semenov\fnref{label3}}
\ead{semenov@math.vsu.ru}
\author[label2]{Fedor Sukochev\fnref{label4}}
\ead{f.sukochev@unsw.edu.au}
\author[label2]{Alexandr Usachev\corref{cor}\fnref{label4}}
\ead{a.usachev@unsw.edu.au}
\author[label2]{Dmitriy Zanin\fnref{label4}}
\ead{d.zanin@unsw.edu.au}

\cortext[cor]{Corresponding Author}

\address[label1]{Voronezh State University, Universitetskaja pl., 1, Voronezh, 394006, Russia.}
\address[label2]{School of Mathematics and Statistics, University of New South Wales, Sydney, 2052, Australia.}

\fntext[label3]{The research was supported by RFBR grant  14-01-00141а.}

\fntext[label4]{The research was supported by the Australian Research Council.}

\begin{abstract}
We introduce a new approach to traces on the principal ideal $\mathcal L_{1,\infty}$ generated by any positive compact operator whose singular value sequence is the harmonic sequence.
Distinct from the well-known construction of J.~Dixmier, the new approach provides the explicit construction of every trace of every operator in $\mathcal L_{1,\infty}$ 
in terms of translation invariant functionals applied to a sequence of restricted sums of eigenvalues.
The approach is based on a remarkable bijection between the set of all traces on $\mathcal L_{1,\infty}$ 
and the set of all translation invariant functionals on $l_\infty$. 
This bijection allows us to identify all known and commonly used subsets of traces 
(Dixmier traces, Connes-Dixmier traces, etc.) in terms of invariance properties of linear functionals on $l_\infty$, 
and definitively classify the measurability of operators in $\mathcal L_{1,\infty}$ in terms of qualified convergence of sums of eigenvalues. 
This classification has led us to a resolution of several open problems (for the class $\mathcal L_{1,\infty}$) from~\cite{CS}.
As an application we
extend Connes' classical trace theorem to positive normalised traces.
\end{abstract}

\begin{keyword}
Singular traces \sep Banach limits \sep Lidskii formula \sep Connes' trace theorem.

\medskip \MSC Primary: 47B10 \sep 47G30 \sep 58B34 \sep 58J42
\end{keyword}

\end{frontmatter}

\section{Introduction}\label{intro}

Let $B(H)$ denote the algebra of all bounded linear operators on a separable Hilbert space $H$.
Denote by $\{\mu(n,A)\}_{n\ge0}$ the sequence of singular values of a compact operator $A \in B(H)$.
Define the principal ideal $\mathcal L_{1,\infty}$ (also termed the weak-$\mathcal L_1$ ideal) of the algebra $B(H)$ by setting
$$\mathcal L_{1,\infty}:= \left\{A \in B(H) \text{ is compact}: \sup_{n \ge 0}  (1+n) \mu(n,A) < \infty\right\}.
$$
A trace on $\mathcal{L}_{1,\infty}$ is a unitarily invariant linear functional on $\mathcal{L}_{1,\infty}$.

We present a new approach to the construction of traces on $\mathcal L_{1,\infty}$, which, in a way, completes the original idea of J.~Dixmier, \cite{D}. Our construction was inspired by that of A.~Pietsch, \cite{P}.
Let $\{\lambda(n,A)\}_{n\ge0}$ be a sequence of eigenvalues of a compact operator $A \in B(H)$, ordered in a such way that the sequence $\{|\lambda(n,A)|\}_{n\ge0}$ is decreasing. 
Observing that, for every operator $A \in \mathcal L_{1,\infty}$, the sequence
$$\left\{ \sum_{k=2^n-1}^{2^{n+1}-2} \lambda(k,A) \right\}_{n \ge 0}$$
is bounded, we construct the functional 
\begin{equation}\label{weight}
\tau(A) = \theta\left(\frac1{\log 2} \left\{ \sum_{k=2^n-1}^{2^{n+1}-2} \lambda(k,A) \right\}_{n \ge 0} \right) , \quad A \in \mathcal{L}_{1,\infty},
\end{equation}
where $\theta$ is a linear translation invariant functional on $l_\infty$
The remarkable fact we show is that this construction provides a linear \emph{bijective} association between traces on $\mathcal L_{1,\infty}$ and linear translation invariant functionals on $l_\infty$.


Before continuing, let us compare the construction~\eqref{weight} with the very well known construction of J.~Dixmier. As Dixmier stated in a letter to the conference ``Singular traces and their applications'' at Luminy, 2012 (see the notes to Chapter 6 of~\cite{LSZ}),
his first idea was to construct a singular (non-normal) trace $A \mapsto t(A)$ on the ideal $\mathcal L_{1,\infty}$ by using the formula
$$t(A):= \theta (\left\{ (n+1) \lambda(n,A) \right\}_{n\ge0}), \quad 0\le A\in \mathcal L_{1,\infty},$$
where $\theta$ is an extended limit on $l_\infty$ (a Hahn-Banach extension to $l_\infty$ of the ordinary limit on the set of convergent sequences $c$). However, Dixmier ``wasn't able to prove the additivity of $t(A)$''.
Following N.~Aronszajn's advice, Dixmier changed the setting from the ideal $\mathcal L_{1,\infty}$ to the larger ideal $\mathcal M_{1,\infty}$ and succeeded.
Being more precise, for an arbitrary translation and dilation invariant extended limit $\omega$ on $l_\infty$ (a precise definition can be found in the subsequent sections), 
J. Dixmier~\cite{D} constructed the weight
\begin{equation}\label{Dix_original_construction}
{\rm Tr}_\omega (A):= \omega \left( \left\lbrace \frac{1}{\log(2+n)} \sum_{k=0}^n\mu(k,A)\right\rbrace _{n\ge0}\right), \quad 0\le A\in \mathcal M_{1,\infty}, 
\end{equation}
where
$$
\mathcal M_{1,\infty}:= \left\{A \in B(H) \text{ is compact}:  \sup_{n \ge 0}  \frac{1}{\log(2+n)} \sum_{k=0}^n\mu(k,A) < \infty\right\} .
$$
The weight ${\rm Tr}_\omega$ extends to a singular trace (called by subsequent authors a Dixmier trace) on $\mathcal M_{1,\infty}$. 
Evidently $\mathrm{Tr}_\omega$ on $\mathcal M_{1,\infty}$ restricts to a positive trace on $\mathcal L_{1,\infty}$. 
It follows directly from Definitions~\ref{Dix_M} and~\ref{Dix_L} below that 
every Dixmier trace on $\mathcal L_{1,\infty}$ extends to a Dixmier trace on $\mathcal M_{1,\infty}$. 
Here we use the term "Dixmier trace on $\mathcal L_{1,\infty}$" for the restriction of a Dixmier trace from $\mathcal M_{1,\infty}$ 
to $\mathcal L_{1,\infty}$ (see Definition~\ref{Dix_L} below).

However, not every positive trace on $\mathcal L_{1,\infty}$ is the restriction 
of a positive trace on $\mathcal M_{1,\infty}$ (see Theorem~\ref{new} below). 
Therefore Dixmier's construction does not generate all traces on $\mathcal L_{1,\infty}$. 
In addition, in Dixmier's construction~\eqref{Dix_original_construction} the extended limit $\omega$ is far from being unique (see e.g.~\cite[Theorem 40]{SUZ1}). As a consequence, we are denied a neat characterisation of traces by using known characterisations of extended limits.
We also remark that the construction~\eqref{Dix_original_construction} was given in terms of singular values, and it has been a non-trivial task to formulate traces in terms of eigenvalues (see~\cite{K} and Lidkii formulas in~\cite{SSZ,LSZ}).


As stated, and shown below, the advantage of the construction~\eqref{weight} is that it is a bijective association between traces on $\mathcal L_{1,\infty}$ and linear translation invariant functionals on $l_\infty$, and that it is formulated in terms of eigenvalues.
This has fundamental consequences for the study of traces on $\mathcal L_{1,\infty}$ and allied topics 
(in particular, those parts of noncommutative geometry which employ singular traces). 
Indeed, this approach has led us to a  complete description of various sets of traces on $\mathcal{L}_{1,\infty}$ and measurable operators in $\mathcal L_{1,\infty}$ introduced by A.~Connes.

As an example, let us recall that Connes observed in~\cite{C}, that in order to ensure that the functional ${\rm Tr}_\omega$ be a trace, 
it is sufficient to only assume in~\eqref{Dix_original_construction} that $\omega$ is a dilation invariant extended limit on $l_\infty$.  
Dixmier's original construction used a dilation and translation invariant extended limit.  Later it was proved in~\cite[Theorem 2]{SUZ1} (see also~\cite[Theorem 9.6.9]{LSZ}) that the set of traces constructed by dilation invariant extended limits coincided with the set of traces constructed using translation and dilation invariant extended limits. Therefore there is no ambiguity in calling the set of all traces generated by dilation invariant extended limits the set of Dixmier traces, and we denote the set by $\mathcal{D}$.
Using~\eqref{weight} we find (see Theorem~\ref{FS1} below) that the set $\mathcal{D}$ is isometric to the set of ``factorisable'' Banach limits,~\cite{Raimi}. 
Recall a Banach limit $\theta$ on $l_\infty$ is ``factorisable'' if it is of the form $\theta = \gamma \circ C$ for an extended limit $\gamma$ and $C : l_\infty \to l_\infty$ the Cesaro operator. 
As a consequence, a compact operator $A \in \mathcal{L}_{1,\infty}$ is measurable with respect to $\mathcal{D}$ 
(that is, it has the same value for every Dixmier trace) 
if and only if its eigenvalue sequence satisfies the condition that
$$
C \left\{ \sum_{k=2^n-1}^{2^{n+1}-2} \lambda(k,A) \right\}_{n \ge 0} \text{ is convergent} .
$$

Combining this characterisation with the fact that the classes of Dixmier traces (on $\mathcal L_{1,\infty}$) 
and normalised fully symmetric functionals on $\mathcal L_{1,\infty}$ coincide (see Corollary~\ref{fs_cor1} below),
we are able to resolve (for the class $\mathcal L_{1,\infty}$) an open problem (iii) stated~\cite[p. 1061]{CS},
concerning the measurability with respect to the class of all normalised fully symmetric functionals.

A smaller subclass of Dixmier traces was suggested by A.~Connes in~\cite[Section IV, 2$\beta$]{C}. 
It was observed that for any extended limit $\gamma$ on $l_{\infty}$ the functional $\omega:=\gamma\circ M$ is dilation invariant. 
Here, the bounded operator $M:l_\infty\to l_\infty$ is given by the formula
$$(Mx)_n=\frac1{\log(2+n)}\sum_{k=0}^n \frac{x_k}{k+1},\ n\ge 1.$$
Let $\mathcal C$ denote the set of traces generated by using a dilation extended limit $\omega$ of the form $\omega:=\gamma\circ M$ in~\eqref{Dix_original_construction}.  
This set of traces was termed ``Connes-Dixmier traces'' in~\cite{LSS}. Evidently $\mathcal C$ is a subset of $\mathcal D$, i.e.~every Connes-Dixmier trace is a Dixmier trace.  
It is a strict subset since it is known that the set of Dixmier and Connes-Dixmier traces are distinct~\cite[Theorem 6.1]{P3}. 
This distinction was studied by A.~Pietsch in a series of three papers~\cite{P1,P4,P3}, where the deep techniques were developed,
which are of a wider interest in the theory of singular traces.
The inclusion $\mathcal C \subsetneq \mathcal D$ was also proved in~\cite[Theorem 2.2]{SUZ2} using a different approach.  
Using~\eqref{weight} we find (see Theorem~\ref{CD} below) the neat classification establishing the (isometric) bijection between the set $\mathcal{C}$ 
and that of Banach limits of the form $\theta = \gamma \circ C^2$ for an extended limit $\gamma$. 
Thus a compact operator $A \in \mathcal{L}_{1,\infty}$ is measurable with respect to $\mathcal{C}$ (having the same value for every Connes-Dixmier trace) 
if and only if its eigenvalue sequence satisfies the condition that
$$
C^2 \left\{ \sum_{k=2^n-1}^{2^{n+1}-2} \lambda(k,A) \right\}_{n \ge 0} \text{ is convergent} .
$$
This leads to a surprising result.  It is known that on the positive cone of $\mathcal L_{1,\infty}$ the notions of Dixmier- and Connes-Dixmier measurable coincide, 
\cite[Corollary 3.9]{LSS}.  See~\cite{CS,C,LSS,LS,LSZ,SUZ2} for properties and concrete examples of Dixmier- and Connes-Dixmier measurable operators. 
Using the above classification and a Tauberian result of G.~Hardy we show that for every operator $A \in \mathcal L_{1,\infty}$ 
(not necessarily positive) the notion of Dixmier- and Connes-Dixmier measurable coincide. 
This result resolves the problem (i) stated in~\cite[p. 1061]{CS} (in the ideal $\mathcal L_{1,\infty}$) in the affirmative.
Whether this result remains true on the larger ideal $\mathcal M_{1,\infty}$ remains unknown.  

Finally, an even smaller subclass of Dixmier traces was considered by many authors (see e.g.~\cite{BF,CPS2}).
Within the set of dilation invariant extended limits of the form $\omega:=\gamma\circ M$
there will be those that satisfy $\omega:=\omega \circ M$. For evident reasons such an extended limit is called an $M$-invariant extended limit. 
Let $\mathcal D_M$ denote the set of traces generated by using an $M$-invariant extended limit $\omega$ of the form $\omega:=\omega \circ M$ in~\eqref{Dix_original_construction}. 
Evidently $\mathcal D_M$ is a subset of $\mathcal C$. 
Dixmier traces from the set $\mathcal D_M$ are used in various important formulae in noncommutative geometry, 
such as (a) Connes' formula for a representative of the Hochschild class  of the Chern character for $(p,\infty)$-summable spectral triples 
(see e.g. \cite[Theorem~7]{CPRS1} and \cite[Theorem 6]{BF}), and (b) the formulae involving heat kernel estimates and generalised $\zeta$-functions residues 
(see e.g. \cite{CPS2,CPRS1,BF,CRSS,SZ}).

Using~\eqref{weight} we prove in Theorem~\ref{M-inv} below that the set $\mathcal{D}_M$ is isometric to the set of Banach limits of the form $\theta = \theta \circ C$ 
(the set of Cesaro invariant Banach limits). 
This is sufficient to show that $\mathcal D_M$ is a strict subset of the set $\mathcal C$ of Connes-Dixmier traces. 
It follows from some further results that a compact operator $A \in \mathcal{L}_{1,\infty}$ is measurable with respect to $\mathcal{D}_M$ 
(having the same value for every $M$-invariant Dixmier trace) if and only if its eigenvalue sequence satisfies the condition that
$$
\lim_{m\to\infty} \liminf_{n\to\infty}C^m\left\{ \sum_{k=2^n-1}^{2^{n+1}-2} \lambda(k,A) \right\}_{n \ge 0} =
 \lim_{m\to\infty} \limsup_{n\to\infty}C^m\left\{ \sum_{k=2^n-1}^{2^{n+1}-2} \lambda(k,A) \right\}_{n \ge 0}
$$
(this result should be compared with that~\cite[Theorem 5, Corollary 13]{SS_JFA}).
We show (see Theorem~\ref{dif2} below) that the set of operators that are measurable with  
respect to $\mathcal D_M$ contains as a strict subset those measurable with respect to the sets traces $\mathcal{D}$ or $\mathcal{C}$.


Going beyond Dixmier traces, the construction~\eqref{weight} gives us a clear path to study the set of all positive normalised traces on $\mathcal L_{1,\infty}$. 
Recall that a trace $\tau$ on $\mathcal{L}_{1,\infty}$ is normalised if $\tau(A) = 1$ for every $0 \leq A \in \mathcal{L}_{1,\infty}$ with $\mu(n,A) = (1+n)^{-1}$, $n \ge 0$. 
The set of positive normalised traces, denoted $\mathcal{PT}$, and the notion of operators measurable with respect to the set $\mathcal{PT}$ is studied here for the first time. 
Using~\eqref{weight} it turns out that there is an isometry between the set $\mathcal{PT}$ and the set of all Banach limits on $l_\infty$. 
This makes the set of positive normalised traces a very natural set of traces to consider. 
Further, the notion of $\mathcal{PT}$-measurability is intricately linked to the classical notion of almost convergence introduced  by G.~G.~Lorentz in~\cite{L}.  
An operator $A \in \mathcal{L}_{1,\infty}$ is measurable with respect to the set $\mathcal{PT}$ of traces (that is, all positive normalised traces 
take the same value on $A$) if and only if 
$$
\left\{ \sum_{k=2^n-1}^{2^{n+1}-2} \lambda(k,A) \right\}_{n \ge 0} \text{ is almost convergent} .
$$
We also prove that the class of Dixmier measurable operators is strictly wider than the class of $\mathcal{PT}$-measurable operators within $\mathcal L_{1,\infty}$ (see Theorem~\ref{dif1} below).


To summarise the above result on traces in symbols, let us introduce
\begin{dfn}\label{meas}
Let $\mathcal A$ be a subset of traces on $\mathcal L_{1,\infty}$. 
The set $\mathcal L_{1,\infty}^{\mathcal A}$ of all $\mathcal A$-measurable elements consists of all elements 
$A\in \mathcal L_{1,\infty}$ such that $\tau(A)$ takes the same value for all $\tau \in \mathcal A$.
\end{dfn}
We have the various sets of normalised positive traces on $\mathcal{L}_{1,\infty}$ 
$$
\mathcal{D}_M \subsetneq \mathcal C \subsetneq \mathcal D \subsetneq  \mathcal{PT},
$$
where $\mathcal D$ is the traditional set of Dixmier traces and, using~\eqref{weight},
these sets are isometric (respectively) to
\begin{multline*}
\{\theta=\theta \circ C : \theta \text{ is an extended limit} \}
\subsetneq 
\{\theta=\gamma \circ C^2 : \gamma \text{ is an extended limit} \} \\
\subsetneq
\{\theta=\gamma \circ C : \gamma \text{ is an extended limit} \}
\subsetneq
\{\theta : \theta \text{ is a Banach limit} \}
\end{multline*}
(see Theorems~\ref{M-inv},~\ref{CD},~\ref{FS1} and Corollary~\ref{BL1} below).
Summing up all the results about measurability we obtain the following chain of inclusions
$$
\mathcal L_{1,\infty}^{\mathcal{PT}}\subsetneq\mathcal L_{1,\infty}^{\mathcal{D}}=\mathcal L_{1,\infty}^{\mathcal{C}}\subsetneq \mathcal L_{1,\infty}^{\mathcal{D}_M}.
$$


As an application, the explicit form of positive normalised traces provided by~\eqref{weight} allows us to extend some of the results of \cite{KLPS}, which studied Connes' trace theorem, to all positive traces. In particular, the construction~\eqref{weight} gives us a unique formula by which we can calculate the positive normalised trace of a compactly supported pseudo-differential operator of order $-d$ using its symbol. It also provides conditions for when a pseudo-differential operator of order $-d$ may have a unique residue calculated by using a positive normalised trace on $\mathcal L_{1,\infty}$. A detailed explanation is given in Section~\ref{NCG} but we sketch the relevant ideas here.

Recall that the original statement of Connes' trace theorem, \cite{C3}, is as follows.
\begin{thm}\label{CTTstatement0}
Every compactly supported classical pseudo-differential operator $A : C_c^\infty(\mathbb R^d) \to C_c^\infty(\mathbb R^d)$ of order $-d$
extends to a compact linear operator belonging to $\mathcal L_{1,\infty}(L_2(\mathbb R^d))$ and 
$${\rm Tr}_\omega(A) = \frac{1}{d (2\pi)^d} {\rm Res}_W(A),$$
where ${\rm Res}_W(A)$ is Wodzicki's (noncommutative) residue of $A$ and 
${\rm Tr}_\omega$ is any Dixmier trace.
\end{thm}
Connes' statement was given for closed manifolds, but it is equivalent to Theorem~\ref{CTTstatement0}. In~\cite{KLPS} (see also~\cite{LPS} and~\cite[Section 11]{LSZ}) generalisations of Connes' trace theorem were given. A wider class of operators, called Laplacian modulated operators, was considered in~\cite{KLPS}.  All pseudo-differential operators of order $-d$ were shown to be Laplacian modulated operators (including of course the smaller set of classical operators).
For the Laplacian modulated operators a vector valued Wodzicki's residue ${\rm Res}$ was defined. It belonged to $l_\infty / c_0$ and extended the Wodzicki's residue ${\rm Res}_W$~\cite[Proposition 6.16]{KLPS}, meaning that if $A$ is a compactly supported classical pseudo-differential operator then ${\rm Res}(A)$ is a scalar and ${\rm Res}(A)= {\rm Res}_W(A)$.

The main result of Section~\ref{NCG} complements Theorem 6.32 from~\cite{KLPS}.
\begin{thm}\label{CTTv2}
A compactly supported Laplacian modulated operator $A$ 
extends to a compact linear operator $A \in \mathcal L_{1,\infty}(L_2(\mathbb R^d))$ and if $\tau$ is a normalised positive trace on $\mathcal L_{1,\infty}(L_2(\mathbb R^d))$, then
$$
\tau(A) = \frac1{(2\pi)^d \log 2}  B \left( \left\{ \int_{\mathbb R^d}\int_{2^{n/d}<|s|\le 2^{(n+1)/d}} p_A(x,s)ds dx \right\}_{n \ge 0} \right),
$$
for a unique Banach limit $B$ (corresponding to the trace $\tau$). Further, the equality
$$
\tau(A)= \frac1{d(2\pi)^d} {\rm Res}(A) 
$$
holds for every positive normalised trace $\tau$ on $\mathcal L_{1,\infty}$
if and only if the sequence
\begin{equation}
\left\{ \int_{\mathbb R^d} \int_{2^\frac{n}d <|s|\le 2^\frac{n+1}d} p_A(x,s) ds\,dx \right\}_{n\ge0}
\end{equation}
is almost convergent (in the sense of Definition~\ref{ac}) to the scalar value $\frac1d \log 2 \cdot {\rm Res}(A)$.  
\end{thm}


The paper and our methods, which are perhaps of a wider interest and applicability, are organised as follows:

1. It is standard to reduce questions concerning traces on $\mathcal L_{1,\infty}$ to questions concerning functionals on its commutative counterpart $l_{1,\infty}$.
In Section~\ref{sf} we study the symmetric and fully symmetric functionals on $l_{1,\infty}$ 
(the commutative counterparts of traces and Dixmier traces on $\mathcal{L}_{1,\infty}$, respectively). 
The new approach to traces mentioned above is initially stated for symmetric functionals. So, we also introduce a universal way to constuct symmetric functionals on $l_{1,\infty}$. 
We prove that continuous symmetric functionals on $l_{1,\infty}$ form a lattice, which nicely complements the result of H.~Lotz~\cite[Theorem 4.3]{Lotz} about symmetric functionals on the function space $L_{1,\infty}$ on the interval $(0,1)$.
Then we show the three-way bijection between: symmetric functionals on the weak-$l_1$ space $l_{1,\infty}$, traces on the corresponding ideal $\mathcal L_{1,\infty}$ of compact operators
and translation invariant linear functionals on the space $l_\infty$ of bounded sequences. 
We also specialize these bijections to the set of all positive normalised traces on $\mathcal L_{1,\infty}$ and Banach limits on $l_\infty$ in Corollary~\ref{BL1}.

2. In Section~\ref{tr} we transfer the results proved in Section~\ref{sf} for symmetric functionals to the noncommutative setting. 
The main result of this section (Theorem~\ref{one-to-one1}) 
introduces the bijection between the set $\mathcal {PT}$ of positive normalised traces on $\mathcal{L}_{1,\infty}$ and Banach limits.
Moreover, this bijection is an order isomorphism and an isometry from the set of all continuous traces on $\mathcal L_{1,\infty}$ 
to the set of all bounded translation invariant linear functionals on $l_\infty$.
Having this powerful result in hand, we 
study geometric properties
of the set $\mathcal {PT}$. In particular, we show that the diameter of the latter set equals $2$ in Corollary~\ref{diam}.
We also characterise extreme points of the set $\mathcal {PT}$ in Theorem~\ref{ext1}. 

3. Section~\ref{subclasses} is devoted to the study of various subclasses of the class of positive normalised traces on $\mathcal{L}_{1,\infty}$.
We characterise Dixmier, Connes-Dixmier traces and the class $\mathcal D_M$ of Dixmier traces generated by $M$-invariant extended limits
in terms of subclasses of Banach limits, as already mentioned. 

4. In Section~\ref{lid} we establish the Lidskii formula for traces on $\mathcal L_{1,\infty}$, which allows one to evaluate a trace using eigenvalues of an operator instead of its singular values.

5. In Section~\ref{Me} we investigate the measurability of operators from $\mathcal L_{1,\infty}$  with respect to different subclasses of positive normalised traces. 
The main result of this section is Theorem~\ref{D=CD}, answering the question about the relationship between the classes of Dixmier- and Connes-Dixmier measurable operators.

6. In the last section we apply our results to pseudo-differential operators.

\section{Preliminaries}
We denote by $L_\infty := L_\infty (0,\infty)$ the space of all (equivalence classes of) real-valued
essentially bounded Lebesgue measurable functions on $(0,\infty)$ equipped with the norm
$$\|x\|_{L_\infty}:=\mathop{\rm esssup}_{t>0} |x(t)|.$$ 
Let $C_0 := C_0(0,\infty)$ denote the subspace of all bounded continuous functions on $(0,\infty)$ that vanish at infinity.

Let $\pi$ be the isometric embedding $\pi:l_{\infty}\to L_\infty$
given by 
\begin{equation}\label{pi}
 \{x_n\}_{n=0}^\infty\stackrel{\pi}{\mapsto}\sum_{n=0}^\infty x_n \chi_{[n,n+1)}.
\end{equation}
For every operator $A\in B(H)$ a generalised singular value function $\mu(A)$ is defined by the formula
$$\mu(t, A)=\inf \{\|Ap\|_\infty : \ p \ \text{is a projection in} \ B(H) \ \text{with} \ {\rm Tr}(1-p)\le t\}.$$
Since $B(H)$ is an atomic von Neumann algebra and traces of all atoms equal to 1, it follows that $\mu(A)$ is a step function for every $A\in B(H)$ (see e.g.~\cite[Chapter 2]{LSZ}).
In particular, for a compact operator $A \in B(H)$ we have that $\mu(n, A)$, $n \geq 0$, is the $n$-th singular value of the operator $A$.

We give the definition of extended limits on functions and sequences.
\begin{dfn}
  A positive linear functional on $L_\infty$ is called an extended limit if it coincides with the ordinary limit on every convergent (at $+\infty$) function.
 A positive linear functional on $l_\infty$ is called an extended limit if it coincides with the ordinary limit on every convergent sequence.
\end{dfn}

\begin{rem}\label{extended_limits}
It is well-known that for every extended limit $\gamma$ and every $x \in l_\infty$ the following inequalities hold
 $$\liminf_{n\to\infty} x_n \le \gamma(x) \le \limsup_{n\to\infty} x_n.$$
Hence, every extended limit is a positive norm-one functional on $l_\infty$.
Moreover, for every $x\in l_\infty$ the following equality holds
$$\{\gamma(x) : \gamma \ \text{is an extended limit}\} = [\liminf_{n\to\infty} x_n , \limsup_{n\to\infty} x_n].$$
Similar statements hold for extended limits on $L_\infty$.
\end{rem}

\begin{dfn}
An extended limit $\omega$ on $L_\infty$ is called dilation invariant if
$$\omega \circ \sigma_\beta  =\omega \ \text{for every} \ \beta>0.$$
Here,
$$(\sigma_\beta x)(t):=x( t/\beta), t>0.$$
\end{dfn} 

Let $S  : l_\infty \to l_\infty$ be the right shift operator defined as follows
$$S(x_0, x_1, x_2, \dots) = (0,x_0, x_1, x_2, \dots)$$
and let $T  : l_\infty \to l_\infty$ be the left shift operator defined as follows
$$T(x_0, x_1, x_2, \dots) = (x_1, x_2, \dots).$$

We also define a dilation operator $\sigma_2  : l_\infty \to l_\infty$ as follows
$$\sigma_2(x_0, x_1, x_2, \dots) = (x_0, x_0, x_1,x_1, x_2, x_2, \dots).$$

It is evident that any left-shift-invariant functional on $l_\infty$ is invariant with respect to the right shift.
The converse statement is false in general. However, the converse statement holds for bounded functionals. 

\begin{prop}\label{bf}
For every bounded linear functional $\theta$ on $l_\infty$ we have 
$\theta=\theta \circ S$ if and only if $\theta=\theta \circ T$.
\end{prop}

\begin{proof}
We show the ``if'' direction; the ``only if'' direction is trivial.
Let $\theta$ be a bounded functional on $l_\infty$ such that $\theta=\theta \circ S$.
We first show that $\theta(1,0,0,\dots)=0$.
Assume contrapositively that $\theta(1,0,0,\dots)=a\neq 0$.
Since $\theta=\theta \circ S$, it follows that 
$$\theta(\underbrace{1,\dots,1}_{n \ \text{times}}, 0,0,\dots)= \sum_{k=1}^n \theta \circ S^k (1,0,0,\dots) =na.$$
Hence, the functional $\theta$ is not bounded.

The obtained contradiction shows that for every bounded functional $\theta$ on $l_\infty$ 
such that $\theta=\theta \circ S$ one has $\theta(1,0,0,\dots)=0$ or equivalently \\$\theta(x_0,0,0,\dots)=0$.

Hence,
\begin{align*}
 \theta(x) = \theta(x_0, x_1, x_2, \dots) - \theta(x_0, 0, 0, \dots) 
 = \theta(0, x_1, x_2, \dots) 
 = \theta(STx) =\theta(Tx).  
\end{align*}
\end{proof}

\begin{dfn}
A linear functional $B$ on  $l_\infty$ is called a Banach limit if
\begin{enumerate}
    \item[(i)] $B\geq0$, that is $B(x) \geq 0$ for $x \geq 0$,
    \item[(ii)] $B(\emm)=1$, where $\emm=(1,1,1,\dots)$,
    \item[(iii)] $B(Sx)=B(x)$ for all $x\in l_\infty$.
\end{enumerate} 
\end{dfn}
Note that, originally, Banach limits were defined to be $T$-invariant~\cite[Chapter II, \S 3, Example 4]{B}.
In view of Proposition~\ref{bf} our definition is equivalent to that of Banach.
We denote the set of all Banach limits by $\mathfrak B$.

The following concept was introduced by G.~G.~Lorentz,~\cite{L}.

\begin{dfn}\label{ac}
 A sequence $x\in l_\infty$ is said to be almost convergent (to $a\in \mathbb R$) if $Bx=a$ for every Banach limit $B$.
\end{dfn}

Denote the set of all almost convergent sequences by $ac$. 
Denote by $ac_0$ the subset of all sequences almost convergent to zero.
The following criterion of almost convergence was proved by Lorentz,~\cite{L}.
\begin{thm}\label{lorentz}
A sequence $x\in l_\infty$ is almost convergent to $a\in \mathbb R$ if and only if
$$
 \lim_{n\to\infty}\frac1n\sum_{k=m}^{m+n-1}x_k=a
$$
uniformly in $m\in{\mathbb N}$. 
\end{thm}

\section{Symmetric functionals on $l_{1,\infty}$}\label{sf}
Denote by $l_{1,\infty}$ the linear space (frequently called the weak $l_1$-space) of all bounded sequences for which the quasi-norm
$$\|x\|_{l_{1,\infty}} = \sup_{n \ge 0} (n+1) x_n^*$$
is finite. Recall that by $x^*$ we denote a decreasing rearrangement of $|x|$.

The following definition should be compared with a similar notion studied in~\cite{DPSSS1,DPSSS2,DPSS}.
\begin{dfn}
 A linear functional $\varphi$ on $l_{1,\infty}$ is called symmetric if
$\varphi(x)=\varphi(y)$ for every $0\le x,y \in l_{1,\infty}$ such that $x^*=y^*$.
\end{dfn}

Denote by 
$$Z := {\rm Lin} \{u-v : 0\le u,v \in l_{1,\infty},  u^*=v^*\}.$$

It should be noticed that every symmetric functional $\varphi$ on $l_{1,\infty}$ vanishes on $Z$. 
In particular, for every symmetric functional $\varphi$ we have $\varphi(x) = \frac12 \varphi(\sigma_2 x)$ for every $x \in l_{1,\infty}$.
We also have that every symmetric functional vanishes on $l_1$,~\cite[Proposition 2.6]{DPSS}.

The following lemma is proved for the case of a Banach symmetric sequence space in~\cite[Proposition 4.2.8]{LSZ}.
Although we state it for a quasi-Banach symmetric sequence space, the proof is exactly the same and is, therefore, omitted.
\begin{lem}\label{phi+}
 For every continuous symmetric functional $\varphi$ on $l_{1,\infty}$ the functional
 $$\varphi_+(x) := \sup_{0\le u \le x} \varphi(u), \ 0\le x \in l_{1,\infty}$$
 satisfies the following properties:
 
 (i) $\varphi_+(x+y) = \varphi_+(x) +\varphi_+(y)$, $x,y \ge 0$;
 
 (ii) $\varphi_+(x^*) = \varphi_+(x)$, $x \ge 0$; \\
 In particular, the functional $\varphi_+$ extends to a positive symmetric functional on $l_{1,\infty}$.
\end{lem}

%
%
%
%
%
%
%
%

Denote by $l_{1,\infty}^*$ the space of all linear functionals on $l_{1,\infty}$.
For $\varphi, \psi \in l_{1,\infty}^*$ define the functionals $\varphi\vee \psi, \varphi\wedge \psi \in l_{1,\infty}^*$ 
by the following formulae
$$(\varphi\vee \psi) (x) = \sup\{\varphi(u) + \psi(v) : \ 0\le u, v \in l_{1,\infty}, x=u+v\}, \ 0\le x \in l_{1,\infty},$$
$$(\varphi\wedge \psi) (x) = \inf\{\varphi(u) + \psi(v) : \ 0\le u, v \in l_{1,\infty}, x=u+v\}, \ 0\le x \in l_{1,\infty}.$$
Observe that $\varphi\vee 0 = \varphi_+$.

%
%

\begin{prop}\label{lattice}
The set of all continuous symmetric functionals on $l_{1,\infty}$ is a lattice 
with respect to the operations $\vee$ and $\wedge$ defined above. 

\end{prop}

\begin{proof}
We shall show that the set of all continuous symmetric functionals on $l_{1,\infty}$ is a sublattice (in the lattice $l_{1,\infty}^*$), 
that is for every continuous symmetric functionals $\varphi$ and $\psi$ the functional $\varphi\vee \psi$ is continuous and symmetric. 

We have
\begin{align*}
(\varphi\vee \psi) (x) &= \sup\{\varphi(u) + \psi(v) : \ 0\le u, v \in l_{1,\infty}, x=u+v\}\\
&= \sup\{\varphi(x-v) + \psi(v) : \ 0\le u, v \in l_{1,\infty}, x=u+v\}\\
&= \varphi(x) + \sup\{(\psi-\varphi)(v) : \ 0\le v \le x \}\\
&= \varphi(x)  + (\psi-\varphi)_+(x).
\end{align*} 

By Lemma~\ref{phi+} the functional $\varphi\vee \psi$ is continuous and symmetric.
Consequently, the set of all continuous symmetric functionals is a sublattice.
\end{proof}

The following linear operator from $l_\infty$ into $l_{1,\infty}$ given by
$$D(x_0, x_1, x_2, \dots) := \log 2 \cdot (\frac{x_0}{2^0}, \underbrace{\frac{x_1}{2^1}, \frac{x_1}{2^1}}_{2 \ \text{times}}, 
\underbrace{\frac{x_2}{2^2},\frac{x_2}{2^2}, \frac{x_2}{2^2}, \frac{x_2}{2^2}}_{4 \ \text{times}}, \dots, \underbrace{\frac{x_n}{2^n}, \dots, \frac{x_n}{2^n}}_{2^n \ \text{times}}, \dots)$$
plays an important role in this paper. The concept of the operator $D$ was suggested by A.Pietsch in~\cite{P}.

It is easy to see that the operator $D$ is continuous from $l_\infty$ into $l_{1,\infty}$
 and that $$\|D\|_{l_\infty \to l_{1,\infty}} = 2 \log 2.$$

The following two lemmas are crucial technical elements in the construction of symmetric functionals on $l_{1,\infty}$.

\begin{lem}\label{aux}
 For every $0\le x, y\in l_{1,\infty}$ we have:

(i) $$\left\{ \sum_{k=2^n-1}^{2^{n+1}-2} x_k^* \right\}_{n \ge 0} \in l_\infty; $$

(ii) $$\left\{ \sum_{k=0}^{2^{n+1}-2} (x_k^*+ y_k^*-(x+y)_k^*) \right\}_{n \ge 0} \in l_\infty.$$
\end{lem}

\begin{proof} 
(i) Since $x^*$ is decreasing, it follows that 
$$\left| \sum_{k=2^n-1}^{2^{n+1}-2} x_k^* \right| \le 2^n x_{2^n-1}^* \le \|x\|_{l_{1,\infty}}, \ n\ge 0.$$

(ii) For every $n\ge 0$ the following estimates hold
$$\sum_{k=0}^{2^{n+1}-2} (x_k^*+ y_k^*)  \le \sum_{k=0}^{2(2^{n+1}-2)+1}(x+y)_k^* \le \sum_{k=0}^{2^{n+1}-2}(x+y)_k^* + 2^{n+1}(x+y)_{2^{n+1}-1}^*.$$
Hence,
$$0\le \sum_{k=0}^{2^{n+1}-2} (x_k^*+ y_k^*-(x+y)_k^*) \le \|x+y\|_{l_{1,\infty}}.$$
\end{proof}

\begin{lem}\label{aux0}
 For every $0\le x \in l_{1,\infty}$ such that $x_n \le \frac{\alpha}{n+1}$, $n\ge0$, for some $\alpha>0$ we have
$$\left| \sum_{k=0}^{n} (x_k^*- x_k) \right| \le \alpha, \ n \ge 0.$$
\end{lem}

\begin{proof} 
For every $n\ge 0$ there exists a subset $A \subset \mathbb N$, such that $|A| = n+1$ and
$$\sum_{k=0}^{n} x_k^* = \sum_{k\in A} x_k.$$
Consequently, we have
\begin{align*}
\sum_{k=0}^{n} x_k^* &= \sum_{k\in A \cap [0,n]} x_k + \sum_{k\in A \cap (n, \infty)} |x_k|\\
&\le \sum_{k=0}^{n} x_k + \sum_{k\in A \cap (n, \infty)} \frac{\alpha}{k+1}\\
&\le \sum_{k=0}^{n} x_k + \sum_{k=n+1}^{2n+1} \frac{\alpha}{k+1}
\le \sum_{k=0}^{n} x_k + \alpha.
\end{align*}
\end{proof}

The following lemma establishes the most important property of the operator $D$.

\begin{lem}\label{aux1}
 For every $0\le x\in l_{1,\infty}$ the sequence
$$x-D \left( \frac1{\log 2} \cdot \left\{ \sum_{k=2^n-1}^{2^{n+1}-2} x_k^* \right\}_{n \ge 0} \right)$$
belongs to $Z$.
\end{lem}

\begin{proof} 

By the definition, for every positive $x\in l_\infty$, we have $x-x^* \in Z$. Therefore, since $Z$ is a linear space it is sufficient to prove the statement for $x=x^*$. 
In this case we have
\begin{align*}
 z&:=x-D \left( \frac1{\log 2} \cdot \left\{ \sum_{k=2^n-1}^{2^{n+1}-2} x_k \right\}_{n \ge 0} \right)\\
&=(0, x_1-\frac{x_1+x_2}{2}, x_2-\frac{x_1+x_2}{2}, \\
& \quad \quad x_3 - \frac{x_3+x_4+x_5+x_6}{4},\dots, x_6-\frac{x_3+x_4+x_5+x_6}4, \dots).
\end{align*}

Note, that $\sum_{i=2^n-1}^{2^{n+1}-2} z_i=0$ for every $n\ge0$.
For every $2^n-1 \le k \le 2^{n+1}-2$ ($n\ge0$), since $x=x^*$, it follows that
$$|z_k|=|x_k - \frac{x_{2^n-1}+x_{2^n}+\dots+x_{2^{n+1}-2}}{2^n}|\le  x_{2^n-1}$$
and so, we also have
$$|\sum_{i=0}^k z_i| = |\sum_{i=2^n-1}^k z_i| \le \sum_{i=2^n-1}^k  x_{2^n-1} \le 2^{n} x_{2^n-1} \le 2 \|x\|_{l_{1,\infty}}.$$
For every $2^n-1 \le k \le 2^{n+1}-2$ ($n\ge0$) we set
$$u_k = \sum_{i=2^n-1}^k z_i \quad  \text{and} \quad 
v_k=\begin{cases}
	0, \ k= 2^n-1;\\
        u_{k-1}, \ 2^n \le k \le 2^{n+1}-2
        \end{cases}.$$
A direct verification shows that $0\le u, v \in l_{1,\infty}$  and
$z=u-v$.
        
Since $\sum_{i=2^n-1}^{2^{n+1}-2} z_i=0$, it follows that $u_{2^{n+1}-2} = 0$. 
Thus, the sequence $v$ is a permutation of $u$.
Hence, $u^*=v^*$  and, so, $z \in Z$.

\end{proof}

The following theorem describes a correspondence between the class of all symmetric functionals on $l_{1,\infty}$
and the class of all $S$-invariant linear functionals on $l_\infty$. 
This correspondence (in a slightly different form) was first found by A.~Pietsch in~\cite{P} (see also~\cite{P2}). 
The idea of studying symmetric functionals on $l_{1,\infty}$ via $S$-invariant linear functionals on $l_\infty$
has become an important motive for the present paper.

\begin{thm}\label{one-to-one}
 (i) For every symmetric functional $\varphi$ on $l_{1,\infty}$ there exists 
a unique $S$-invariant linear functional $\theta=\varphi\circ D$ on $l_\infty$ such that
\begin{equation}\label{formula}
 \varphi(x) = \theta\left(\frac1{\log 2} \cdot \left\{ \sum_{k=2^n-1}^{2^{n+1}-2} x_k^* \right\}_{n \ge 0} \right) , \quad x\ge0.
\end{equation}
(ii) For every linear functional $\theta$ on $l_\infty$ such that $\theta=\theta \circ S$ 
the functional $\varphi$ defined by the formula~\eqref{formula} extends by linearity to a symmetric functional on $l_{1,\infty}$.
\end{thm}

\begin{proof}
 (i) Let $\varphi$ be a symmetric functional on $l_{1,\infty}$. Set $\theta:=\varphi\circ D$.
Due to the linearity of $D$, the functional $\theta$ is linear on $l_\infty$. We prove that $\theta$ is invariant under the operator $S$.

We firstly prove an auxiliary fact that $\varphi(\frac12\sigma_2(Dx)-DSx)=0$. Indeed,
\begin{align*}
 y:&= \frac1{\log 2} \cdot (\frac12\sigma_2(Dx)-DSx)\\
&=(\frac{x_0}{2}, \frac{x_0}{2}; \frac{x_1}{2^2}, \frac{x_1}{2^2},\frac{x_1}{2^2},\frac{x_1}{2^2};  \dots)
-(0; \frac{x_0}{2^1}, \frac{x_0}{2^1}; \frac{x_1}{2^2}, \frac{x_1}{2^2},\frac{x_1}{2^2},\frac{x_1}{2^2}; \dots)\\
&=(\frac{x_0}{2};0,\frac{x_1}{2^2}-\frac{x_0}{2^1};0,0,0, \frac{x_2}{2^3}-\frac{x_1}{2^2}; \dots).
\end{align*}

For $u:=(\frac{x_0}{2};0,\frac{x_1}{2^2};0,0,0, \frac{x_2}{2^3}; \dots)$ and $v:=(0;0,\frac{x_0}{2^1};0,0,0, \frac{x_1}{2^2}; \dots)$
we have $y=u-v$ and $u^*=v^*$. So, $\varphi(y)=0$, since $\varphi$ is symmetric.

Next, using the fact that $\varphi(\frac12\sigma_2(Dx))=\varphi(DSx)$ and that $\varphi$ is symmetric, we have
$$ \theta(Sx)=\varphi(DSx)=\varphi(\frac12\sigma_2(Dx))=\varphi(Dx)=\theta(x).$$
Hence, $\theta$ is an $S$-invariant linear functional on $l_\infty$.

By Lemma~\ref{aux1} for every $0 \le x\in l_{1,\infty}$ we have 
$$x-D \left( \frac1{\log 2} \cdot \left\{ \sum_{k=2^n-1}^{2^{n+1}-2} x_k^* \right\}_{n \ge 0} \right) \in Z.$$

Using the fact that $\varphi$ vanishes on $Z$, we derive
$$\varphi(x)=(\varphi \circ D)\left( \frac1{\log 2} \cdot \left\{ \sum_{k=2^n-1}^{2^{n+1}-2} x_k^* \right\}_{n \ge 0} \right)$$
and, since $\theta=\varphi\circ D$, the representation~\eqref{formula} is proved. We now show that this representation is unique.

Let $\theta_1$ be an $S$-invariant linear functional on $l_\infty$ such that
$$\varphi(x) = \theta_1 \left( \frac1{\log 2} \cdot \left\{ \sum_{k=2^n-1}^{2^{n+1}-2} x_k^* \right\}_{n \ge 0} \right) , 0\le x \in l_{1,\infty}.$$

According to the definition of the functional $\theta$, we have
$$\theta(x)=\varphi(Dx) = \theta_1 \left( \frac1{\log 2} \cdot \left\{ \sum_{k=2^n-1}^{2^{n+1}-2} (Dx)_k^* \right\}_{n \ge 0} \right) , 0\le x \in l_\infty.$$

To complete the proof it is sufficient to verify that the expression on the right equals to $\theta_1(x)$. 
Set 
$$y_n := \log 2 \cdot \sum_{k=0}^n x_k - \sum_{k=0}^{2^{n+1}-2} (Dx)_k^*, \ n \ge 0.$$
By the definition of the operator $D$, we have
$$\log 2 \cdot \sum_{k=0}^n x_k = \sum_{k=0}^{2^{n+1}-2} (Dx)_k, \ \forall n\ge0$$ 
and so
$$y_n = \sum_{k=0}^{2^{n+1}-2} (Dx)_k - \sum_{k=0}^{2^{n+1}-2} (Dx)_k^*, \ n \ge 0.$$
Since $\|D\|_{l_\infty \to l_{1,\infty}} = 2 \log 2$, it follows that 
$$(Dx)_n \le \frac{2\log 2}{n+1}\|x\|_{l_\infty}, \ n \ge 0.$$
Hence, for every $0\le x \in l_\infty$ the sequence $Dx$ satisfies the assumptions of Lemma~\ref{aux0}. 
Therefore, we conclude that $y\in l_\infty.$
Obviously, we have 
$$y_{n+1}-y_n = \log 2 \cdot x_{n+1} - \sum_{k=2^{n+1}-1}^{2^{n+2}-2} (Dx)_k^*, \ n \ge 0$$
and, using the definition of the operator $S$, 
$$\log 2 \cdot x-\left\{ \sum_{k=2^n-1}^{2^{n+1}-2} (Dx)_k^* \right\}_{n \ge 0} = y-Sy.$$

Since $\theta_1$ is an $S$-invariant linear functional, it follows that $\theta_1(Sy-y)=0$ and
$$\theta_1(x)=\theta_1 \left( \frac1{\log 2} \cdot \left\{ \sum_{k=2^n-1}^{2^{n+1}-2} (Dx)_k^* \right\}_{n \ge 0} \right) , 0\le x \in l_\infty$$
and $\theta(x)=\theta_1(x)$ for every $0\le x \in l_\infty.$ Hence, $\theta_1(x)=\theta(x)$ for all $x\in l_\infty$.

(ii) Let $\theta=\theta \circ S$ be a linear functional on $l_\infty$. It is clear that the functional $\varphi$
given by~\eqref{formula}
is positive homogeneous on the positive cone of $l_{1,\infty}$. We shall prove its additivity on the positive cone of $l_{1,\infty}$. 
For $0\le x,y \in l_{1,\infty}$
we set
$$z_n := \sum_{k=0}^{2^{n+1}-2} (x_k^*+ y_k^*-(x+y)_k^*), \ n\ge 0.$$
By Lemma~\ref{aux}, $z \in l_\infty$. An immediate computation yields
$$(z-Sz)_n = \sum_{k=2^n-1}^{2^{n+1}-2} (x_k^*+ y_k^*-(x+y)_k^*), \ n\ge 0.$$
Due to $S$-invariance of $\theta$ we have $\theta(Sz-z)=0$ and hence
\begin{equation}\label{lin}
 \theta\left( \left\{ \sum_{k=2^n-1}^{2^{n+1}-2} (x_k^*+ y_k^*-(x+y)_k^*)\right\}_{n \ge 0} \right)=0.
\end{equation}

Using the definition~\eqref{formula}, the equality~\eqref{lin} and the linearity of $\theta$, for every $0\le x,y \in l_{1,\infty}$ we obtain
\begin{align*}
\varphi(x+y)&=\theta\left( \frac1{\log 2} \cdot \left\{ \sum_{k=2^n-1}^{2^{n+1}-2} (x+y)_k^* \right\}_{n \ge 0} \right)\\
&=\theta\left( \left\{ \frac1{\log 2} \cdot \sum_{k=2^n-1}^{2^{n+1}-2} (x_k^*+ y_k^*)\right\}_{n \ge 0} \right)\\
&=\theta\left( \left\{ \frac1{\log 2} \cdot \sum_{k=2^n-1}^{2^{n+1}-2} x_k^*\right\}_{n \ge 0} \right)+ 
\theta\left( \left\{ \frac1{\log 2} \cdot \sum_{k=2^n-1}^{2^{n+1}-2}y_k^*)\right\}_{n \ge 0} \right)\\
&=\varphi(x)+\varphi(y). 
\end{align*}

So, the functional $\varphi$ given by~\eqref{formula} is positive homogeneous and additive on the positive cone of $l_{1,\infty}$.
Hence, it extends by linearity to the whole space $l_{1,\infty}$.

It is clear, that for every $0\le x,y \in l_{1,\infty}$ such that $x^*=y^*$ we have $\varphi(x)=\varphi(y)$.
Hence, the formula~\eqref{formula} defines a symmetric functional on $l_{1,\infty}$.
\end{proof} 

Now we specialise the result of Theorem~\ref{one-to-one} to the case of positive symmetric functionals and
to the case of continuous symmetric functionals.

\begin{cor}\label{pos}
(i) For every positive symmetric functional $\varphi$ on $l_{1,\infty}$ there exists 
a unique positive linear functional $\theta=\theta \circ S$ on $l_\infty$ such that~\eqref{formula}
 holds.

(ii) For every positive linear functional $\theta=\theta \circ S$ on $l_\infty$ 
the functional $\varphi$ defined by the formula~\eqref{formula} extends by linearity to a positive symmetric functional on $l_{1,\infty}$.
\end{cor}

\begin{proof}
In view of Theorem~\ref{one-to-one} we only need to show that the functional $\varphi$ is positive if and only if $\theta= \varphi \circ D$ is positive.
If $\varphi\ge0$, then $\theta=\varphi \circ D\ge 0$, since $D$ is a positive operator from $l_\infty$ into $l_{1,\infty}$.
If $\theta\ge 0$, then the positivity of $\varphi$ follows from~\eqref{formula}.
\end{proof}

\begin{cor}\label{cont}
(i) For every continuous symmetric functional $\varphi$ on $l_{1,\infty}$ there exists 
a unique continuous linear functional $\theta=\theta \circ S$ on $l_\infty$ such that~\eqref{formula}
 holds.

(ii) For every continuous linear functional $\theta=\theta \circ S$ on $l_\infty$ 
the functional $\varphi$ defined by the formula~\eqref{formula} extends by linearity to a continuous symmetric functional on $l_{1,\infty}$.
\end{cor}

\begin{proof}
 In view of Theorem~\ref{one-to-one} we only need to show that the functional $\varphi$ is continuous if and only if $\theta= \varphi \circ D$ is continuous.

If $\varphi$ is continuous on $l_{1,\infty}$, then $\theta= \varphi \circ D$ is continuous on $l_\infty$, 
since $D : l_\infty \to l_{1,\infty}$ is continuous.

By Proposition~\ref{lattice} the set of all continuous symmetric functionals is a lattice. 
%
%
%
Since every vector lattice is a linear hull of its positive elements (see e.g.~\cite[Theorem 1.1.1]{MN}), the statement follows from Corollary~\ref{pos}.
\end{proof}

%
%
%

Recall that a symmetric functional on $l_{1,\infty}$ is normalised if
$\varphi(\{\frac1{n+1}\}_{n \ge 0}) = 1$.
The following corollary describes the correspondence between the set of all positive normalised symmetric functionals on $l_{1,\infty}$ 
and the set $\mathfrak B$ of all Banach limits.
 
\begin{thm}\label{BL}
 A linear functional $\varphi$ is a positive normalised symmetric functional on $l_{1,\infty}$ 
 if and only if $B=\varphi \circ D \in \mathfrak B$.
\end{thm}
\begin{proof}
For every positive symmetric functional $\varphi$ on $l_{1,\infty}$ by Corollary~\ref{pos} 
we obtain that the functional $B=\varphi \circ D$ is a positive translation invariant functional on $l_\infty$,
that is $B$ proportional to a Banach limit. We only need to check that $B(\emm)=1$.
Indeed,
\begin{align*}
 1&= \varphi(\{\frac1{n+1}\}_{n \ge 0})
 =B\left(\frac1{\log 2} \cdot \left\{ \sum_{k=2^n-1}^{2^{n+1}-2} \frac1{k+1} \right\}_{n \ge 0} \right)\\
 &=B\left(\frac1{\log 2} (\log 2 +o(1)) \right)
 =B(\emm).
\end{align*}
The assertion has been proved. The proof of the ``only if'' part is similar.
\end{proof}

The rest of the section is devoted to the study of a subset of all symmetric functionals.
The following notion has been studied in many papers including~\cite{DPSS,DPSSS1,DPSSS2,KS,KSS}.
\begin{dfn}
 A linear functional $\varphi$ on $l_{1,\infty}$ is called fully symmetric if
$\varphi(x) \le \varphi(y)$ for every $0\le x,y \in l_{1,\infty}$ such that $x\prec\prec y$, 
that is $\sum_{k=0}^n x_k^* \le \sum_{k=0}^n y_k^*$ for every $n\ge 0$.
\end{dfn}

Observe that every fully symmetric functional is automatically positive and symmetric.

Define the Ces\`{a}ro operator $C: l_\infty \to l_\infty$ as  follows
$$(Cx)_n = \frac1{n+1} \sum_{k=0}^{n} x_k, \ n \in \mathbb N.$$

Before we state the main result, describing fully symmetric functionals on $l_{1,\infty}$ 
in terms of factorisable Banach limits (that is a functional of the form $\gamma \circ C$ for some extended limit $\gamma$~\cite{Raimi})
, we prove an auxiliary technical lemma.

\begin{lem}\label{aux2}
 Let $x \in l_{1,\infty}$ be such that $|x_n| \le \frac{\alpha}{n+1}$, $n\ge 0$ (for some $\alpha>0$).
 If $\sum_{k=0}^n x_k \le 0$, $n\ge 0$, then $\varphi(x) \le 0$ for every fully symmetric functional $\varphi$ on $l_{1,\infty}$.
\end{lem}

\begin{proof} Let $\varphi$ be a fully symmetric functional on $l_{1,\infty}$. 
Without loss of generality, $\varphi(\{\frac1{n+1}\}_{n \ge 0})=1.$

 Set $y_n =x_n +\frac{\alpha}{n+1}$. Observe that $y_n\ge 0$ for every $n\ge 0$.
 By Lemma~\ref{aux0}, we have
 $$\sum_{k=0}^n y_k^* \le \sum_{k=0}^n y_k +\alpha, \ n\ge0.$$
 Since $\sum_{k=0}^n x_k \le 0$, it follows that
 $$\sum_{k=0}^n y_k^* \le \sum_{k=0}^n \frac{\alpha}{k+1}+\alpha.$$
 Setting $z_0 := \alpha$, $z_k := 0$ ($k\ge1$), we obtain
 $$\sum_{k=0}^n y_k^* \le \sum_{k=0}^n (\frac{\alpha}{k+1}+z_k).$$
Hence, using the fact that $\varphi(\{\frac1{n+1}\}_{n\ge0})=1$ we have
 $$\varphi(y) \le \varphi(\{\frac{\alpha}{n+1}\}_{n\ge0}) + \varphi(z) = \alpha.$$
To obtain the last equality we used the fact that every symmetric functional on $l_{1,\infty}$ is singular, 
that is it vanishes on finitely supported sequences 
(alternatively, one can use Theorem~\ref{BL} and the fact that Banach limits vanish on finitely supported sequences).

Hence, $\varphi(x) \le 0$.
\end{proof}

%

It was established in~\cite[Theorem 11]{KSS} that every normalised fully symmetric functional can be written in the form~\eqref{Dix_original_construction}
with some dilation invariant extended limit $\omega$.
The following theorem shows that every normalised fully symmetric functional can be written in the form~\eqref{formula}
with some factorisable Banach limit $\theta$.

\begin{thm}\label{FS}
(i) For every normalised fully symmetric functional $\varphi$ on $l_{1,\infty}$ there exists 
an extended limit $\gamma$ on $l_\infty$ such that~\eqref{formula}
 holds for $\theta=\gamma \circ C$

(ii) For every extended limit $\gamma$ on $l_\infty$ and $\theta=\gamma \circ C$ 
the functional $\varphi$ defined by the formula~\eqref{formula} extends by linearity to a normalised fully symmetric functional on $l_{1,\infty}$.
\end{thm}

\begin{proof} 
(i) Let $\varphi$ be a normalised fully symmetric functional on $l_{1,\infty}$. 
Hence, $\varphi$ is a normalised positive symmetric functional on $l_{1,\infty}$.
By Theorem~\ref{BL}, there exists a unique Banach limit $B$ on $l_\infty$ given by the formula $B = \varphi \circ D$. 
We need to show that $B$ can be expressed as $B =\gamma \circ C$ for some extended limit $\gamma$ on $l_\infty$.

For every $x\in l_\infty$ and $2^m-1 \le n \le 2^{m+1}-2$ ($m \ge 0$) it follows from the definition of the operator $D$ that
\begin{align*}
\sum_{k=0}^n (Dx)_k &= \sum_{k=0}^{2^m-2} (Dx)_k + \sum_{k=2^m-1}^n (Dx)_k = \sum_{k=0}^{m-1} x_k + \frac{n-2^m+2}{2^m}x_m \\
&\le \max\left\{ \sum_{k=0}^{m-1} x_k, \sum_{k=0}^{m} x_k\right\}. 
\end{align*}

In particular, it follows from the above, that for every $x\in l_\infty$ such that $Cx \le 0$, we have $CDx \le 0$ and, by Lemma~\ref{aux2} 
(we can apply this lemma since $\|D\|_{l_\infty \to l_{1,\infty}} = 2 \log 2$ and so, $(Dx)_n \le \frac{2\log 2}{n+1}\|x\|_{l_\infty}$), 
since $\varphi$ is positive, it follows that $\varphi(Dx) \le 0$.
Consequently, $B(x) =\varphi(Dx) \le 0$ for every $x\in l_\infty$ such that $Cx \le 0$.

We have $C(x - \sup_{n\ge0} (Cx)_n)\le 0$ for every $x\in l_\infty$. 
Hence,
$B(x- \sup_{n\ge0} (Cx)_n)\le 0$ and 
\begin{equation}\label{eq11}
B(x) \le \sup_{n\ge0} (Cx)_n 
\end{equation}
 for every $x \in l_\infty$.

Denote by $C(l_\infty)$ the range of the Ces\`aro operator $C : l_\infty \to l_\infty$.
Since $Cx=Cy \Leftrightarrow x=y$, the operator $C : l_\infty \to C(l_\infty)$ is a bijection.

Define a linear functional $\gamma$ on $C(l_\infty)$ by the formula 
$$\gamma(x):= B(C^{-1}x), \ x \in C(l_\infty).$$

For every $x \in C(l_\infty)$, using~\eqref{eq11}, we have that 
$$\gamma(x) = B(C^{-1}x) \le \sup_{n\ge0} (CC^{-1}x)_n  = \sup_{n\ge0} x_n.$$
Using Hahn-Banach theorem we extend $\gamma$ from $C(l_\infty)$ to $l_\infty$ preserving the inequality $\gamma \le \sup.$
For every $x\in l_\infty$, we have 
$$-\gamma(x) =\gamma(-x) \le \sup_n (-x_n) = -\inf_n x_n.$$
Hence,
\begin{equation}\label{gl}
 \inf_n x_n \le \gamma(x) \le \sup_n x_n.
\end{equation}

Hence $\gamma$ is a positive norm-one functional on $l_\infty$.

It remains to show that $\gamma$ is an extended limit.
Due to~\eqref{gl}, it is sufficient to show that $\gamma$ vanishes on every sequence with finite support.

 For every sequence $x\in l_\infty$ a direct verification shows that $(C^{-1}x)_0 =x_0$ and
 $$(C^{-1}x)_n = (n+1) x_n - nx_{n-1}, \ n \ge 1.$$
Hence, if $x\in l_\infty$ is a sequence with finite support we conclude, 
 that $C^{-1}x \in l_\infty$ is also a sequence with finite support. 
Hence, $\gamma(x) = B(C^{-1}x)=0$, since $B$ is a Banach limit.

Consequently, $\gamma$ is an extended limit.
Finally, for every $x\in l_\infty$ we trivially have $Cx \in C(l_\infty)$ and
$$B(x) = B(C^{-1}Cx)=\gamma(Cx).$$

(ii) Let $\theta=\gamma \circ C$ for some extended limit $\gamma$ on $l_\infty$. 
Since $C\circ S- C : l_\infty \to c_0$ and since $\gamma$ is an extended limit, it follows that $\theta$ is $S$-invariant.
Hence, by Theorem~\ref{one-to-one} the formula~\eqref{formula} defines a symmetric functional $\varphi$ on $l_{1,\infty}$.
We only need to show that the functional $\varphi$ is fully symmetric.

Let $0\le x,y \in l_{1,\infty}$ be such that $x\prec\prec y$.
We have 
$$\left(C\left\{ \sum_{k=2^n-1}^{2^{n+1}-2} x_k^* \right\}_{n \ge 0}\right)_m 
= \frac1{m+1} \sum_{n=0}^m \left\{ \sum_{k=2^n-1}^{2^{n+1}-2} x_k^* \right\}_{n \ge 0}
= \frac1{m+1} \sum_{n=0}^{2^{m+1}-2} x_n^*$$

and (since $x\prec\prec y$ )
$$\left(C\left\{ \sum_{k=2^n-1}^{2^{n+1}-2} x_k^* \right\}_{n \ge 0}\right)_m \le \left(C\left\{ \sum_{k=2^n-1}^{2^{n+1}-2} y_k^* \right\}_{n \ge 0}\right)_m, \ \forall \ m\ge0.$$

Next by~\eqref{formula}, using the linearity of $\varphi$ and since $\theta=\gamma \circ C$, we write
\begin{align*}
\varphi(x-y) &=  \frac1{\log 2} \theta\left(\left\{ \sum_{k=2^n-1}^{2^{n+1}-2} x_k^* \right\}_{n \ge 0} - \left\{ \sum_{k=2^n-1}^{2^{n+1}-2} y_k^* \right\}_{n \ge 0}\right)\\
&= \frac1{\log 2} \gamma \left(C\left\{ \sum_{k=2^n-1}^{2^{n+1}-2} x_k^* \right\}_{n \ge 0} - C\left\{ \sum_{k=2^n-1}^{2^{n+1}-2} y_k^* \right\}_{n \ge 0}\right)\\
&\le 0,
\end{align*}
where the latter inequality is due to the positivity of $\gamma$ (see Remark~\ref{extended_limits}).
Consequently, $\varphi(x) \le \varphi(y)$, that is, $\varphi$ is a fully symmetric functional. 
\end{proof}

\section{Traces on $\mathcal L_{1,\infty}$}\label{tr}

%
%
%

In this section we extend the construction in Section~\ref{sf} 
to the case of traces on the ideal $\mathcal L_{1,\infty}$.

By ${\rm diag}$ we denote the diagonal operator in $B(H)$ with respect to any fixed basis in $H$.
The following theorem constructs traces on $\mathcal L_{1,\infty}$,
using translation invariant functionals on $l_\infty$.
As for the commutative counterpart, this construction was suggested by A.Pietsch in~\cite{P}.

\begin{thm}\label{one-to-one1}
(i) For every trace $\tau$ on $\mathcal L_{1,\infty}$ there exists 
a unique $S$-invariant linear functional $\theta=\tau \circ {\rm diag} \circ D$ on $l_\infty$ such that
\begin{equation}\label{formula1}
 \tau(A) = \theta\left(\frac1{\log 2} \left\{ \sum_{k=2^n-1}^{2^{n+1}-2} \mu(k,A) \right\}_{n \ge 0} \right) , A\ge0.
\end{equation}
(ii) For every linear functional  $\theta$ on $l_\infty$ such that $\theta=\theta \circ S$ 
the functional $\tau$ defined by the formula~\eqref{formula1} extends by linearity to a trace on $\mathcal L_{1,\infty}$.

In particular, there exists a bijective linear correspondence between the set of all traces on $\mathcal L_{1,\infty}$ and
all $S$-invariant linear functionals on $l_\infty$.
\end{thm}

\begin{proof}
(i) Let $\tau$ be a trace on $\mathcal L_{1,\infty}$. It follows from the results of Dykema, Figiel, Weiss and Wodzicki~\cite{DFWW} 
(see~\cite[Theorem 4.4.1]{LSZ} and~\cite[p. 26]{LSZ} for the detailed explanation) that 
 $$\tau(A) = \tau({\rm diag} (\mu(A)))$$
for every trace on $\mathcal L_{1,\infty}$ and every $0 \le A \in \mathcal L_{1,\infty}$.
 
 Since the functional $\tau \circ {\rm diag}$ is a  symmetric functional on $l_{1,\infty}$,
 it follows from Theorem~\ref{one-to-one} that there exist a unique $S$-invariant linear functional $\theta=(\tau \circ {\rm diag}) \circ D$ on $l_\infty$ such that
\begin{equation}
 (\tau \circ {\rm diag})(x) = \theta\left(\frac1{\log 2} \left\{ \sum_{k=2^n-1}^{2^{n+1}-2} x_k^* \right\}_{n \ge 0} \right) , 0\le x \in l_{1,\infty}.
\end{equation}

Since for every $A \in \mathcal L_{1,\infty}$ the sequence $\{ \mu(n,A) \}_{n \ge 0} \in l_{1,\infty}$ we obtain the assertion of the first part of the theorem.


(ii) The proof that the functional given by~\eqref{formula1} is a weight is similar to that of Theorem~\ref{one-to-one} and therefore omitted.
The functional~\eqref{formula1} is obviously symmetric on $\mathcal L_{1,\infty}$ and, therefore, is a trace on $\mathcal L_{1,\infty}$ by~\cite[Theorem 2.7.4]{LSZ}.
\end{proof}

Recall that $\mathcal L_{1,\infty}$ is equipped with the quasi-norm
$$
\| A \|_{\mathcal L_{1,\infty}} := \sup_{n \geq 0} (n+1)\mu(n,A) , \quad A \in \mathcal L_{1,\infty} .
$$
Now we specialize Theorem~\ref{one-to-one1} to the cases of positive and (quasi-norm) continuous traces.
It will be done in a similar fashion to that of Section~\ref{sf}, that is we first prove the result for positive traces,
then using the lattice property of the set of all continuous traces we deduce the ``continuous'' case from the ``positive'' one.  Recall that a trace $\tau$ on $\mathcal L_{1,\infty}$ is normalised if $\tau \circ {\rm diag} ( \{ (n+1)^{-1} \}_{n\ge 0} = 1$.

\begin{cor}\label{BL1}
(i) For every positive normalised trace $\tau$ on $\mathcal L_{1,\infty}$ there exists 
a unique Banach limit $B=\tau \circ {\rm diag} \circ D$ such that~\eqref{formula1} holds with $\theta=B$.

(ii) For every Banach limit $B$ 
the functional $\tau$ defined by the formula~\eqref{formula1} (with $\theta=B$) extends by linearity to a positive normalised trace on $\mathcal L_{1,\infty}$.
\end{cor}

\begin{proof}
(i) Since for every trace $\tau$ on $\mathcal L_{1,\infty}$ the functional $\tau \circ {\rm diag}$ is a  symmetric functional on $l_{1,\infty}$,
 it follows from Theorem~\ref{BL} that the  functional 
 $$B=\tau \circ {\rm diag} \circ D$$ 
 is a Banach limit.
 
(ii) Let $B$ be a Banach limit, that is a positive normalised $S$-invariant functional on $l_\infty$.
Hence, by Theorem~\ref{one-to-one1} the functional $\tau$ defined by the formula~\eqref{formula1} extends by linearity to a trace on $\mathcal L_{1,\infty}$.
Its positivity clearly follows from the positivity of $B$. We also have
\begin{align*}
 \tau({\rm diag}(\{\frac1{n+1}\}_{n \ge 0}))&=B\left(\frac1{\log 2} \cdot \left\{ \sum_{k=2^n-1}^{2^{n+1}-2} \frac1{k+1} \right\}_{n \ge 0} \right)\\
 &=B\left(\frac1{\log 2} (\log 2 +o(1)) \right)
 =B(\emm)=1.
\end{align*}
\end{proof}

The proof of the following theorem is similar to that of Proposition~\ref{lattice} and therefore omitted. 
It can be consider as a noncommutative counterpart of the result of H.~Lotz~\cite[Theorem 4.3]{Lotz} 
concerning symmetric functionals on the function space $L_{1,\infty}$ on the interval $(0,1)$
\begin{thm}\label{lattice1}
The set of all continuous traces on $\mathcal L_{1,\infty}$ is a lattice. 
\end{thm}

The proof of the following corollary is similar to that of Corollary~\ref{cont}. 

\begin{cor}\label{cont1}
(i) For every continuous trace $\tau$ on $\mathcal L_{1,\infty}$ there exists 
a unique continuous linear functional $\theta=\theta \circ S$ on $l_\infty$ such that~\eqref{formula1}
 holds.

(ii) For every continuous linear functional $\theta=\theta \circ S$ on $l_\infty$ 
the functional $\tau$ defined by the formula~\eqref{formula1} extends by linearity to a trace on $\mathcal L_{1,\infty}$.
\end{cor}

Now we show that not every trace on $\mathcal L_{1,\infty}$ extends to a trace on $\mathcal M_{1,\infty}$.  We recall that $\mathcal M_{1,\infty}$ is a Banach space when equipped with the norm
$$
\| A \|_{\mathcal M_{1,\infty}} := \sup_{n \geq 0} \frac1{\log(n+2)} \sum_{k=0}^n \mu(k,A) , \quad A \in \mathcal M_{1,\infty} .
$$
For a given $n \in \mathbb N$ we consider a compact operator $A_n$ such that 
$$\mu(A_n) := \sup 2^{-j-n(i+2)^2} \chi_{[0, 2^{j+n(i+2)^2})},$$
where the supremum is taken over all $i, j \in \mathbb N$ such that $j\le i$.

\begin{lem}\label{lm}
 For every $n \in \mathbb N$ we have $A_n \in \mathcal L_{1,\infty}$. 
 Moreover, $\|A_n\|_{\mathcal L_{1,\infty}} \le 1$, however $\|A_n\|_{\mathcal M_{1,\infty}} \le 2/n$.
\end{lem}
\begin{proof} Since the sequence $\{\mu(m,A_n)\}_{m\ge 0}$ is piecewise constant, it follows that the supremum of $m \cdot \mu(m,A_n)$ attains at the right endpoints of the intervals of constancy. That is,
$$\sup_m m \cdot \mu(m,A_n) = \sup_{i, j \in \mathbb N : j\le i} 2^{j+n(i+2)^2} \mu(2^{j+n(i+2)^2},A_n) =1$$
and, so, $\|A_n\|_{\mathcal{L}_{1,\infty}}\leq 1$.

For every $m\in \mathbb N$ select $k\ge0$ such that 
$m\in[2^{n(k+2)^2},2^{n(k+3)^2}),$ then
\begin{align*}
\sum_{l=0}^m\mu(l,A_n)&\leq\sum_{l=0}^{2^{n(k+3)^2}-1}\mu(l,A_n)= 1+\sum_{i=0}^k \sum_{l=2^{n(i+2)^2}}^{2^{n(i+3)^2}-1}\mu(l,A_n) \\
&= 1+\sum_{i=0}^k i \leq (k+2)^2. 
\end{align*}

%
%

Due to the choice of $k$ we have that $(k+2)^2 \le \frac1n \log_2 (m+2)$.
Hence,
$$ \sum_{l=0}^m \mu(l,A_n) \le \frac1n \log_2 (m+2) \le \frac2n \log (2+m).$$
\end{proof}

Recall that $\mathcal {PT}$ denotes the set of all positive normalised traces on $\mathcal L_{1,\infty}$.

\begin{lem}\label{max}  We have
$$\sup_{\tau\in\mathcal{PT}}\tau(A_n)\geq\frac1{2\log 2}.$$
\end{lem}
\begin{proof} 
It was proved in~\cite{S} that for every $x\in l_\infty$, we have
$$\sup_{B\in \mathfrak B} Bx = \lim_{k\to \infty} \sup_{m\ge 0} \frac1k \sum_{i=m}^{k+m-1} x_i, \ x\in l_\infty.$$
Moreover, the supremum on the left-hand side is attained.

From Theorem~\ref{BL1} for every positive operator $A\in \mathcal L_{1,\infty}$ we have
\begin{align*}
\sup_{\tau\in \mathcal{PT}} \tau(A) &= \sup_{B\in \mathfrak B} B\left(\frac1{\log 2} \left\{ \sum_{i=2^n-1}^{2^{n+1}-2} \mu(i,A) \right\}_{n \ge 0} \right)\\
 &= \frac1{\log 2} \lim_{k\to \infty} \sup_{m\ge 0} \frac1k \sum_{i=2^m-1}^{2^{k+m}-2} \mu(i,A). 
\end{align*}

For fixed $n\in \mathbb N$ and for every $k\in \mathbb N$ we set $m=n(k+2)^2$. We obtain
$$\sup_{\tau\in\mathcal{PT}}\tau(A)\geq\frac1{\log 2}\liminf_{k\to\infty}\frac1k\sum_{l=2^{n(k+2)^2}}^{2^{k+n(k+2)^2-1}}\mu(l,A).$$

Due to the definition of $\mu(A_n)$ we have that
$$\sum_{l=2^{n(k+2)^2}}^{2^{k+n(k+2)^2-1}}\mu(l,A_n)=\frac{k}{2}.$$
The assertion follows immediately.
\end{proof}

The following result shows that the set of all traces on $\mathcal L_{1,\infty}$ is strictly larger 
than that of the restrictions to $\mathcal L_{1,\infty}$ of traces on $\mathcal M_{1,\infty}$.

\begin{thm}\label{new}
 There is a positive trace on $\mathcal L_{1,\infty}$ which does not extend to a positive trace on $\mathcal M_{1,\infty}$.
\end{thm}

\begin{proof}
 Consider the operators $A_n$ constructed above. By Lemma~\ref{lm} all $A_n$ belong to $\mathcal L_{1,\infty}$, however $\|A_n\|_{\mathcal M_{1,\infty}} \le 2/n$.
 By Lemma~\ref{max} there are positive normalised traces $\tau_n$ on $\mathcal L_{1,\infty}$ such that $\tau_n(A_n)\ge \frac1{2\log 2}$.
 
 Set $\tau = \sum_{n=1}^\infty n^{-2}\tau_{n^3}$. It is clear, that $\tau$ is a positive trace on $\mathcal L_{1,\infty}$.
 Suppose that $\tau$ extends to a positive trace on $\mathcal M_{1,\infty}$ 
 (by $\|\tau\|_{\mathcal M_{1,\infty} \to \mathbb C}$ we denote its norm) , then
 $$\frac1{2\log 2} \le \tau_{n^3}(A_{n^3}) \le n^2 \tau(A_{n^3}) 
 \le n^2 \|\tau\|_{\mathcal M_{1,\infty} \to \mathbb C} \cdot \|A_{n^3}\|_{\mathcal M_{1,\infty}}
 \le \frac2n \|\tau\|_{\mathcal M_{1,\infty} \to \mathbb C}.
$$

Letting $n\to\infty$, we obtain a contradiction. 
Thus, $\tau$ does not extend to a positive trace on $\mathcal M_{1,\infty}$.
\end{proof}

%
%
%

Now we are about to characterise the bijection between the set of traces on $\mathcal L_{1,\infty}$ 
and the set of shift invariant functionals on $l_\infty$ established in Theorem~\ref{one-to-one1}.
We first need the following lemma, which is of interest in its own right.

\begin{lem}\label{lattice2}
The set of all continuous $S$-invariant linear functionals on $l_\infty$ is a sublattice in the lattice $l_\infty^*$
with respect to operations $\vee$ and $\wedge$ defined 
by the following formulas
$$(f\vee g) (x) = \sup\{f(u) + g(v) : \ 0\le u, v \in l_\infty, x=u+v\}, \ 0\le x \in l_\infty,$$
$$(f\wedge g) (x) = \inf\{f(u) + g(v) : \ 0\le u, v \in l_\infty, x=u+v\}, \ 0\le x \in l_\infty.$$
\end{lem}

\begin{proof}
By Proposition~\ref{bf} a continuous functional on $l_\infty$ is $S$-invariant if and only if it is $T$-invariant.
So, it is sufficient to prove that the set of all continuous $T$-invariant linear functionals on $l_\infty$ is a sublattice in $l_\infty^*$.
Note that $l_\infty^*$ is lattice, since it is a dual of a Banach lattice (see~e.g.~\cite[Section 1.a]{LT2}).
 It is shown in~\cite[Lemma 1]{SS1} that the operator $T^*$ preserves the operation $\vee$ ans, similarly, the operation $\wedge$.
Hence, the set of all continuous $T$-invariant linear functionals on $l_\infty$ is a lattice. 
\end{proof}

Corollary~\ref{cont1} yields that there is bijection between the set of all continuous traces on $\mathcal L_{1,\infty}$
and the set of all continuous $S$-invariant linear functionals on $l_\infty$. 
The following theorem specifies this correspondence.

\begin{thm}\label{i}
The mapping $i$ from the set of all continuous traces on $\mathcal L_{1,\infty}$ 
to the set of all continuous $S$-invariant linear functionals on $l_\infty$ given by
$$ i(\tau) = \tau \circ {\rm diag} \circ D$$
is an order isomorphism and isometry.
\end{thm}

\begin{proof}
Due to the positivity of an operator $D$, it is clear that $i$ is positive.
It is also clear that the inverse of $i$ given by~\eqref{formula1} is positive.
Hence, $i$ is a positive bijection with the positive inverse, 
that is $i$ is an order isomorphism (see e.g.~\cite[1.a]{LT2}).

Every order isomorphism preserves the lattice structure (see e.g.~\cite[1.a]{LT2}), that is
$$i : \tau_1 \vee \tau_2 \to i(\tau_1) \vee i(\tau_2).$$

In particular, $i(|\tau|)= |i(\tau)|.$ Hence, it is sufficient to prove that $i$ preserves the norm of every positive trace.

If $\tau$ is a positive trace on $\mathcal L_{1,\infty}$, 
then $i(\tau) = \tau \circ {\rm diag} \circ D$ is positive $S$-invariant linear functionals on $l_\infty$. 
In particular, $\tau \circ {\rm diag} \circ D$ is proportional to a Banach limits and, so, vanishes on $c_0$.

By Theorem~\ref{one-to-one1}, we have
$$\tau(A) = (\tau \circ {\rm diag} \circ D) \left(\frac1{\log 2} \left\{ \sum_{k=2^n-1}^{2^{n+1}-2} \mu(k,A) \right\}_{n \ge 0} \right), 
\ 0\le A \in \mathcal L_{1,\infty}.$$

Hence,
\begin{align*}
\|\tau \|_{\mathcal L_{1,\infty}^*} &=  \left|\tau({\rm diag} \{\frac1{n+1}\}_{n \ge 0}) \right|\\
 &=\left| (\tau \circ {\rm diag} \circ D) \left(\frac1{\log 2} \cdot \left\{ \sum_{k=2^n-1}^{2^{n+1}-2} \frac1{k+1} \right\}_{n \ge 0} \right) \right|\\
 &=\left| (\tau \circ {\rm diag} \circ D) \left(\frac1{\log 2} (\log 2 +o(1)) \right)\right|\\ 
 &=| (\tau \circ {\rm diag} \circ D) (\emm) |= \|i(\tau)\|_{l_\infty^*}.
\end{align*}

%
Hence, $\|i(\tau)\|_{l_\infty^*} = \|\tau \|_{\mathcal L_{1,\infty}^*}$ for every continuous trace $\tau$.
Consequently, the mapping $i$ is an isometry.
\end{proof} 

 The following theorem describes the correspondence
 between extreme points of $\mathcal {PT}$ and extreme points of $\mathfrak B$.

\begin{thm}\label{ext1}
 A trace $\tau \in ext(\mathcal {PT})$ if and only if 
$B = \tau \circ {\rm diag} \circ D \in ext\mathfrak B$.
\end{thm}

\begin{proof}
 Let $\tau \in \mathcal {PT}$ and let $B= \tau \circ {\rm diag} \circ D \in ext \mathfrak B$. Suppose that 
 \begin{equation}\label{e1}
 \tau = \frac12(\tau_1+\tau_2) \ \text{ on } \mathcal L_{1,\infty},
 \end{equation}
for some $\tau_1, \tau_2 \in \mathcal {PT}$.

It follows from Corollary~\ref{BL1}, 
that $B := \tau \circ {\rm diag} \circ D$, $B_1 := \tau_1 \circ {\rm diag} \circ D$, $B_2 := \tau_2 \circ {\rm diag} \circ D$ are Banach limits.
Moreover, we obtain from~\eqref{e1} that $B = \frac12(B_1+B_2)$ on $l_\infty$ and, due to the assumption $B\in ext \mathfrak B$ we have $B = B_1=B_2$.
Hence, formula~\eqref{formula1} yields $\tau = \tau_1 = \tau_2$ and $\tau \in ext (\mathcal {PT})$.

Let now $\tau \in ext (\mathcal {PT})$ and let $B= \tau \circ {\rm diag} \circ D \in \mathfrak B$. Suppose 
 \begin{equation}\label{e2}
  B = \frac12(B_1+B_2) \ \text{ on } l_\infty,
 \end{equation}
for some $B_1, B_2 \in \mathfrak B$.

Due to Corollary~\ref{BL1} applying the formula~\eqref{formula1} to 
Banach limits $B_1$ and $B_2$ yields positive normalised traces $\tau_1,\tau_2$.
It follows from~\eqref{e2} that $\tau = \frac12(\tau_1+\tau_2)$ on $l_{1,\infty}$ and, so $\tau = \tau_1 = \tau_2$.
Hence, $B = B_1=B_2$ and $B\in ext \mathfrak B$.
\end{proof}

It was shown in~\cite{SS1} that every sequence $B_i$, $i\ge 1$, of distinct extreme points of $\mathfrak B$ spans the space $l_1$ of all summable sequences.
The following theorem is an analogue of this result for positive normalised traces on $\mathcal L_{1,\infty}$.
\begin{thm}\label{4.11}
 Let $\tau_k \in ext(\mathcal {PT})$, $k \in\mathbb N$, be a sequence of distinct elements.
 For every $\{c_k\}_{k\ge0} \in l_1$ we have
 $$\|\sum_{k=0}^\infty c_k \tau_k\|_{\mathcal L_{1,\infty}^*} = \sum_{k=0}^\infty |c_k|.$$
\end{thm}

\begin{proof} 
 The assertion follows from the fact that the lattices $\mathcal {PT}$ and $\mathfrak B$ are order isomorphic and isometric
 (Theorem~\ref{i}) and the corresponding result for Banach limits~\cite[Theorem 4]{SS1}.
\end{proof}

Define the diameter of the set $\mathcal {PT}$ in $\mathcal L_{1,\infty}^*$ as follows
$$
 d(\mathcal {PT},\mathcal L_{1,\infty}^*)=\sup\limits_{\tau_1, \tau_2\in \mathcal {PT}}\|\tau_1 - \tau_2 \|_{\mathcal L_{1,\infty}^*}.
$$

The set $\mathcal {PT}$ is a subset of the unit sphere of $\mathcal L_{1,\infty}^*$ and
so, its diameter can not exceed $2$. 
The following result shows, in particular, that $d(\mathcal {PT},\mathcal L_{1,\infty}^*)=2$.
It is a straightforward consequence of Theorem~\ref{4.11}.

\begin{cor}\label{diam}
 For every $\tau_1, \tau_2 \in ext(\mathcal {PT})$ such that $\tau_1 \neq \tau_2$
 we have $$\|\tau_1 - \tau_2 \|_{\mathcal L_{1,\infty}^*}=2.$$
\end{cor}

Although the norm of a difference of two exteme points of $ \mathcal {PT}$ equals 2, it is not always attained.
We state this fact rigorously in Corollary~\ref{c5} below. To this end we need some preparations.

It is well-known (see e.g.~\cite[Theorem 9]{Kakutani} and \cite[Chapter 16]{Carothers}) that by the Riesz representation theorem every positive normalised linear functional $l$ on $l_\infty$ 
can be be written in the following form
$$l(x) = \int_{\beta \mathbb N} x(p) d\nu(p), \ \forall \ x \in l_\infty$$
where $\nu$ is a measure on $\beta \mathbb N$, the Stone-\v{C}ech compactification of $\mathbb N$.
Moreover, according to~\cite[436J, 436K]{Fremlin_MT} this measure is Radon.

It was proved in~\cite[Proposition 2.5]{CC} that there exist $B_1, B_2 \in ext(\mathfrak B)$ (with corresponding probability measures $\nu_1$ and $\nu_2$)
such that ${\rm supp} \ \nu_1 \subset {\rm supp} \ \nu_2$ (the support of $\nu$ is well-defined, since the measures are Radon).

\begin{thm}\label{non}
There exist $B_1, B_2 \in ext(\mathfrak B)$ such that 
$|B_1x - B_2x| < 2$ for every $x \in l_\infty$ with $\|x\|_{l_\infty} \le 1$.
\end{thm}

\begin{proof}

By~\cite[Proposition 2.5]{CC} there exist $B_1, B_2 \in ext(\mathfrak B)$ such that ${\rm supp} \ \nu_1 \subset {\rm supp} \ \nu_2$, 
where $\nu_1$ and $\nu_2$ are the probability measures on $\beta \mathbb N$ corresponding to $B_1$ and $B_2$, that is
$$B_1x = \int_{\beta \mathbb N} x(p) d\nu_1(p), \ B_2x = \int_{\beta \mathbb N} x(p) d\nu_2(p), \  \ x \in l_\infty.$$

Denote $S_1 := {\rm supp} \ \nu_1$ and $S_2 := {\rm supp} \ \nu_2$. We have $S_1 \subset S_2$.
Fix $x \in l_\infty$ with $\|x\|_{l_\infty} \le 1$. Define the closed sets 
$$C := \{p \in \beta \mathbb N : x(p) \ge0\}, \ D := \{p \in \beta \mathbb N : x(p) \le0\}.$$

We have
\begin{align*}B_1x - B_2x &= \int_{S_1 \cap C} x(p) d\nu_1(p) + \int_{S_1 \cap D} x(p) d\nu_1(p) \\
&- \int_{S_2 \cap C} x(p) d\nu_2(p) - \int_{S_2 \cap D} x(p) d\nu_2(p).
\end{align*}

Assume that $B_1x - B_2x=2$. We obtain
\begin{align*}
 2 &= B_1x - B_2x \le \int_{S_1 \cap C} x(p) d\nu_1(p) - \int_{S_2 \cap D} x(p) d\nu_2(p)\\
  &\le \|x\|_{l_\infty} (\nu_1(S_1 \cap C)+\nu_2(S_2 \cap D)) \le 2,
\end{align*}
since $\nu_1$ and $\nu_2$ are probability measures.

Therefore, $\nu_1(S_1 \cap C)=\nu_2(S_2 \cap D)=1$. 
Since the support of a measure is the smallest closed set of a full measure, it follows that $S_1 \cap C=S_1$ and $S_2 \cap D=S_2$.

Since $S_1 \subset S_2$, it follows that 
$$(S_1 \cap C) \cap D = S_1 \cap D = (S_1 \cap S_2) \cap D = S_1 \cap (S_2 \cap D) =S_1 \cap S_2 = S_1.$$

Consequently, $x=0$ on $S_1$, the support of $\nu_1$. Hence, $B_1x=0$ and $B_1x - B_2x < 2$,
which contradicts the assumption.
\end{proof}

\begin{cor}\label{c5}
 There exist $\tau_1, \tau_2 \in ext(\mathcal {PT})$ such that 
$|\tau_1(A) - \tau_2(A)| < 2$ for every $A \in \mathcal L_{1,\infty}$ with $\|A\|_{\mathcal L_{1,\infty}} \le 1$.
\end{cor}

\begin{proof}
By Theorem~\ref{non} there exist $B_1, B_2 \in ext(\mathfrak B)$ such that 
$|B_1x - B_2x| < 2$ for every $x \in l_\infty$ with $\|x\|_{l_\infty} \le 1$.

Let $\tau_1, \tau_2 \in \mathcal {PT}$ be traces corresponding to $B_1$ and $B_2$ (by Theorem~\ref{one-to-one1}).
Theorem~\ref{ext1} yields $\tau_1, \tau_2 \in ext(\mathcal {PT})$.

By formula~\eqref{formula1} we obtain
\begin{align*}
|\tau_1(A) - \tau_2(A)|= \left| (B_1-B_2)\left(\frac1{\log 2} \left\{ \sum_{k=2^n-1}^{2^{n+1}-2} \mu(k,A) \right\}_{n \ge 0} \right)\right|.
\end{align*}

Since 
$$\left\| \frac1{\log 2} \left\{ \sum_{k=2^n-1}^{2^{n+1}-2} \mu(k,A) \right\}_{n \ge 0} \right\|_{l_\infty} \le 1$$
for every $\|A\|_{\mathcal L_{1,\infty}} \le 1$,
the assertion follows from Theorem~\ref{non}.

\end{proof}

It follows from~\cite{Talagrand} (see also~\cite{KM,Nillsen}) that the set $ext(\mathfrak B)$ is not $\sigma(l_\infty^*, l_\infty)$-closed.
As a straightforward consequence of Theorem~\ref{ext1} we obtain that the set $ext(\mathcal {PT})$  is not $\sigma(\mathcal L_{1,\infty}^*, \mathcal L_{1,\infty})$-closed.
It is known that the set $\mathfrak B$ is convex and $\sigma(l_\infty^*, l_\infty)$-compact. Hence, by the Krein-Milman theorem we have
$$\mathfrak B = \overline{conv (ext(\mathfrak B))}^{\sigma(l_\infty^*, l_\infty)},$$
where $conv$ denotes the convex hull of the set. 
It is also easy to see that the set $\mathcal {PT}$ is convex and $\sigma(\mathcal L_{1,\infty}^*, \mathcal L_{1,\infty})$-compact
and by the Krein-Milman theorem
$$\mathcal {PT} = \overline{conv (ext(\mathcal {PT}))}^{\sigma(\mathcal L_{1,\infty}^*, \mathcal L_{1,\infty})},$$
that is for every $\tau \in \mathcal {PT}$ there is a net $\tau_\alpha \in conv (ext(\mathcal {PT}))$ 
such that $\tau_\alpha \mathop{\longrightarrow}\limits^{w^*} \tau.$
The following theorem show that the previous statement fails if one changes nets to sequences.
\begin{thm}\label{as_far1}
 Let $\tau \in \mathcal D_M$. 
 The set $conv(ext(\mathcal {PT}))$ does not contain a sequence which is $\sigma(\mathcal L_{1,\infty}^*, \mathcal L_{1,\infty})$-convergent to $\tau$.
\end{thm}

\begin{proof}
 Assume the contrary, that is, there exists a sequence \\$\tau_n \in conv(ext(\mathcal {PT}))$ such that 
 $$\tau_n \mathop{\longrightarrow}\limits^{w^*} \tau.$$
 
 Corollary~\ref{BL1} yields that the functionals $B=\tau \circ {\rm diag} \circ D$ and $B_n=\tau_n \circ {\rm diag} \circ D$ are Banach limits.
 By Theorem~\ref{ext1} we also have that $B_n \in conv(ext \mathfrak B)$.
 
 For every $x\in l_\infty$ we have
 $$|B_n(x)-B(x)| = |\tau_n ({\rm diag} (Dx))-\tau ({\rm diag} (Dx))|.$$
 Since ${\rm diag} (Dx) \in \mathcal L_{1,\infty}$, the right-hand side tends to zero.
 Hence,
 $$B_n \mathop{\longrightarrow}\limits^{w^*} B$$ which contradicts the result of~\cite[Theorem 12]{SS1} (in the view of Theorem~\ref{M-inv} below).
\end{proof}

Answering the question of R.~G.~Douglas, C.~Chou proved that there exists a Banach limit which is not representable 
as a convex linear combination of countably many elements from ${\rm ext} \mathfrak B$~\cite[Proposition 3.2]{CC}.
This result was strengthened in~\cite{SS1}, by showing that any Ces\`aro invariant Banach limit ($B=B\circ C$) has this property.
The following theorem is an analogue of this result for Dixmier traces generated by $M$-invariant extended limits (that is, for traces from $\mathcal D_M$).
\begin{thm}\label{as_far2}
 Let $\tau_k \in ext(\mathcal {PT})$, $k \in\mathbb N$.
 For every $\tau \in \mathcal D_M$
 we have
 $${\rm dist} (\tau, conv \{ \tau_k \}) =2.$$
\end{thm}

\begin{proof} 
 The assertion follows from the fact that lattices $\mathcal {PT}$ and $\mathfrak B$ are order isomorphic and isometric
 (Theorem~\ref{i}) and the corresponding result for Banach limits~\cite[Theorem 13]{SS1} (in the view of Theorem~\ref{M-inv} below).
\end{proof}

\section{Subclasses of positive normalised traces}\label{subclasses}

In this section we describe various classes of Dixmier traces on $\mathcal L_{1,\infty}$, which are of importance in noncommutative geometry (see e.g.~\cite{BF,CPRS1,CPS2,SZ},
in terms of the functional $\theta$ from~\eqref{formula1}.

\begin{dfn}\label{Dix_M}
 A trace $\tau$ on $\mathcal M_{1,\infty}$ is called a Dixmier trace if
it is a linear extension of the weight
$${\rm Tr}_\omega (A):= \omega \left( t \mapsto \frac{1}{\log(1+t)} \int_0^t \mu(s,A) ds \right), \quad 0\le A\in \mathcal M_{1,\infty}, $$
for some dilation invariant extended limit $\omega$ on $L_\infty$.
\end{dfn}


This definition is equivalent to the original construction~\eqref{Dix_original_construction} of Dixmier.  

\begin{dfn}\label{Dix_L}
A trace $\tau$ on $\mathcal L_{1,\infty}$ is called a Dixmier trace if 
it is a restriction of a Dixmier trace on $\mathcal M_{1,\infty}$, 
that is linear extension of a weight
$${\rm Tr}_\omega (A):= \omega \left( t \mapsto \frac{1}{\log(1+t)} \int_0^t \mu(s,A) ds \right), \quad 0\le A\in \mathcal L_{1,\infty}, $$
for some dilation invariant extended limit $\omega$ on $L_\infty$.
\end{dfn}

Despite the fact (Theorem~\ref{new}) that not every positive trace on $\mathcal L_{1,\infty}$ extends to a positive trace on $\mathcal M_{1,\infty}$,
it follows directly from Definitions~\ref{Dix_M} and~\ref{Dix_L} that 
every Dixmier trace on $\mathcal L_{1,\infty}$ extends to a Dixmier trace $\tau$ on $\mathcal M_{1,\infty}$.

In Section~\ref{sf} we have discussed fully symmetric functionals on $l_{1,\infty}$.
Now we define fully symmetric functionals on $\mathcal L_{1,\infty}$ 
and study their relation to Dixmier traces (on $\mathcal L_{1,\infty}$).

\begin{dfn}
 A linear functional $\varphi$ on $\mathcal L_{1,\infty}$ is called fully symmetric if
$\varphi(A) \le \varphi(B)$ for every $0\le A,B \in \mathcal L_{1,\infty}$ such that $A\prec\prec B$, 
that is $\sum_{k=0}^n \mu(k,A) \le \sum_{k=0}^n \mu(k,B)$ for every $n\ge 0$.
\end{dfn}

The following interesting result is proved in~\cite{KSS}.

\begin{thm}\label{41}
 A trace on $\mathcal M_{1,\infty}$ is a Dixmier trace if and only if it is a normalised fully symmetric functional on $\mathcal M_{1,\infty}$.
\end{thm}

A natural question arises: Does a similar result hold for the ideal $\mathcal L_{1,\infty}$?

In Theorem~\ref{fs_ext} we show that every fully symmetric functional on $\mathcal L_{1,\infty}$
extends to a fully symmetric functional on $\mathcal M_{1,\infty}$.
Using this powerful result we show below that the class of Dixmier traces on $\mathcal L_{1,\infty}$ 
coincides with the class of all normalised fully symmetric functionals on $\mathcal L_{1,\infty}$.

We shall use the classical G.~G.~Lorentz and T.~Shimogaki result~\cite[Theorem 1]{LorShi}.

\begin{thm}\label{Lor-Shi}
 If $g_1, g_2, f$ are positive locally integrable functions such that $g_1 +g_2 \prec \prec f$, then  
there exist positive functions $f_1$ and $f_2$
such that $f =f_1+f_2$ and $g_1 \prec \prec f_1$, $g_2 \prec \prec f_2$.
\end{thm}

\begin{thm}\label{fs_ext}
 Every fully symmetric functional on $\mathcal L_{1,\infty}$
extends to a fully symmetric functional on $\mathcal M_{1,\infty}$.
\end{thm}

\begin{proof}
 Let $\varphi$ be a fully symmetric functional on $\mathcal L_{1,\infty}$.
 Set
\begin{equation}\label{def}
 \varphi'(A) := \inf \left\{ \varphi(B) : 0\le B \in \mathcal L_{1,\infty}, A \prec \prec B \right\} , \ 0 \le A \in \mathcal M_{1,\infty}.
\end{equation}

It is clear that $\varphi'$ is a positive homogeneous functional on the positive cone of $\mathcal M_{1,\infty}$.
We shall show that $\varphi'$ is additive on the positive cone of $\mathcal M_{1,\infty}$.

By Theorems 3.3.3 and 3.3.4 from~\cite{LSZ} we have
\begin{equation}\label{sum}
A_1 \oplus A_2 \prec \prec A_1 + A_2 \prec \prec 2 \sigma_{1/2} \mu(A_1 \oplus A_2), 
\end{equation}
where $\oplus$ is the direct sum operation defined as in~\cite[Definition 2.4.3]{LSZ} 
and $\sigma_{1/2}$ is a dilation operator. 
Due to the definition of the direct sum operation we have
$$\varphi (B_1 + B_2)=\varphi (B_1 \oplus B_2)$$
every positive $B_1, B_2 \in \mathcal L_{1,\infty}$.

Hence, using the properties of fully symmetric functionals and~\eqref{sum} we infer that
$$\varphi' (A_1 + A_2)=\varphi' (A_1 \oplus A_2).$$
for every positive $A_1, A_2 \in \mathcal M_{1,\infty}$.

It should be pointed out that the Lorentz-Shimogaki's result is ``commutative''
and can not be applied in the noncommutative setting.
However, it can be applied if one changes the sum of operators to their direct sum.

Fix positive $A_1, A_2 \in \mathcal M_{1,\infty}$.
By Lorentz-Shimogaki's result for positive operators $A_1, A_2, B$ such that $A_1 \oplus A_2 \prec \prec B$ 
there exist positive operators $B_1$ and $B_2$
such that $B =B_1+B_2$ and $A_1 \prec \prec B_1$, $A_2 \prec \prec B_2$.

Hence, due to the property of infimum we obtain
\begin{align*}
\varphi'(A_1 + A_2) &=\varphi'(A_1 \oplus A_2)\\
&= \inf \left\{ \varphi(B) : 0\le B \in \mathcal L_{1,\infty}, A_1 \oplus A_2 \prec \prec B \right\}\\
&\ge \inf \left\{ \varphi(B_1+B_2) : 0\le B_1, B_2 \in \mathcal L_{1,\infty}, A_1 \prec \prec B_1, A_2 \prec \prec B_2 \right\}\\
&= \inf \left\{ \varphi(B_1) : 0\le B_1 \in \mathcal L_{1,\infty}, A_1 \prec \prec B_1 \right\}\\
&+\inf \left\{ \varphi(B_2) : 0\le B_2 \in \mathcal L_{1,\infty}, A_2 \prec \prec B_2 \right\}\\
&=\varphi'(A_1) + \varphi'(A_2).
\end{align*}

Now we prove the converse inequality.

Due to the definition of $\varphi'$~\eqref{def} for every $A_1, A_2 \in \mathcal M_{1,\infty}$ and every $\varepsilon>0$
there are $0 \le B_1, B_2 \in \mathcal L_{1,\infty}$ such that $A_1 \prec \prec B_1$, $A_2 \prec \prec B_2$ and
$$\varphi(B_1) \le \varphi'(A_1) + \varepsilon \ \text{and} \ \varphi(B_2) \le \varphi'(A_2) + \varepsilon.$$

Due to the choice of $B_1, B_2$ and~\eqref{sum}, we obtain
$$A_1 \oplus A_2 \prec \prec B_1 \oplus B_2 \prec \prec B_1 + B_2.$$

Hence,
$$\varphi'(A_1 + A_2) = \varphi'(A_1 \oplus A_2) \le \varphi(B_1 + B_2) \le \varphi'(A_1) + \varphi'(A_2) + 2 \varepsilon.$$

%

Since $\varepsilon>0$ can be chosen arbitrary small, it follows that 
$$\varphi'(A_1 + A_2) \le \varphi'(A_1) + \varphi'(A_2).$$

Consequently, the functional $\varphi'$ is positive homogeneous and additive on the positive cone of $\mathcal M_{1,\infty}$.
So, it extends to a linear functional of $\mathcal M_{1,\infty}$.
Due to the construction~\eqref{def} $\varphi'$ is a fully symmetric functional on $\mathcal M_{1,\infty}$
which coincides with $\varphi$  on $\mathcal L_{1,\infty}$.
\end{proof}

\begin{cor}\label{fs_cor1}
 The class of all Dixmier traces on $\mathcal L_{1,\infty}$ and that of all normalised fully symmetric functionals on $\mathcal L_{1,\infty}$ coincide.
\end{cor}

\begin{proof}
It follows directly from the construction of Dixmier trace (on $\mathcal L_{1,\infty}$) 
that every Dixmier trace is a normalised fully symmetric functional on $\mathcal L_{1,\infty}$.

Conversely, by Theorem~\ref{fs_ext} every normalised fully symmetric functional $\varphi$ on $\mathcal L_{1,\infty}$ 
extends to a normalised fully symmetric functional on $\mathcal M_{1,\infty}$,
that is, due to Theorem~\ref{41}, to Dixmier trace, say ${\rm Tr}_\omega$, on $\mathcal M_{1,\infty}$.
The restriction of the Dixmier trace ${\rm Tr}_\omega$ to $\mathcal L_{1,\infty}$ coincide with the original functional $\varphi$.
However, this restriction is a Dixmier trace on $\mathcal L_{1,\infty}$, by Definition~\ref{Dix_L}.
\end{proof}

Due to Corollary~\ref{fs_cor1} we can use Theorem~\ref{FS} to completely characterise 
Dixmier traces on $\mathcal L_{1,\infty}$ in terms of factorisable Banach limits, 
which were introduced by Raimi in 1980,~\cite{Raimi}.

\begin{thm}\label{FS1}
(i) For every Dixmier trace $\tau$ on $\mathcal L_{1,\infty}$ there exists a factorizable Banach limit $\theta$ such that~\eqref{formula1} holds. 
In other words, there exists 
an extended limit $\gamma$ on $l_\infty$ such that~\eqref{formula1}
 holds for $\theta=\gamma \circ C$

(ii) For every extended limit $\gamma$ on $l_\infty$ and $\theta=\gamma \circ C$ 
the functional $\tau$ defined by the formula~\eqref{formula1} extends by linearity to a Dixmier trace on $\mathcal L_{1,\infty}$.
\end{thm}

\begin{proof}
(i) Let $\tau$ be a trace on $\mathcal L_{1,\infty}$. It follows from the results of Dykema, Figiel, Weiss and Wodzicki~\cite{DFWW} 
(see~\cite[p. 26]{LSZ} for the detailed explanation) that 
 $$\tau(A) = \tau({\rm diag} (\mu(A)))$$
 for every trace on $\mathcal L_{1,\infty}$ and every $0 \le A \in \mathcal L_{1,\infty}$.
 
 Since for every Dixmier trace $\tau$ on $\mathcal L_{1,\infty}$ the functional $\tau \circ {\rm diag}$ is a normalised fully symmetric functional on $l_{1,\infty}$,
 it follows from Theorem~\ref{FS} that there exist an extended limit $\gamma$ on $l_\infty$ such that for $\theta=\gamma \circ C$ one has
\begin{equation*}
 (\tau \circ {\rm diag})(x) = \theta\left(\frac1{\log 2} \left\{ \sum_{k=2^n-1}^{2^{n+1}-2} x_k^* \right\}_{n \ge 0} \right) , 0\le x \in l_{1,\infty}.
\end{equation*}

Since for every positive $A \in \mathcal L_{1,\infty}$ the sequence $\{ \mu(n,A) \}_{n \ge 0} \in l_{1,\infty}$, we obtain 
$$\tau(A) = \tau({\rm diag} (\mu(A))) = \theta\left(\frac1{\log 2} \left\{ \sum_{k=2^n-1}^{2^{n+1}-2} \mu(k,A) \right\}_{n \ge 0} \right),$$
which proves the assertion of the first part of the theorem.

(ii) A direct verification shows that for every extended limit $\gamma$ on $l_\infty$ 
the functional $\theta=\gamma \circ C$ is a Banach limit. 
Hence, by Corollary~\ref{BL1} the functional $\tau$ defined by the formula~\eqref{formula1} 
extends by linearity to a positive normalised trace on $\mathcal L_{1,\infty}$.
We have
\begin{equation}\label{eq56}
\begin{aligned}
\tau(A) &= (\gamma \circ C)\left(\frac1{\log 2} \left\{ \sum_{k=2^n-1}^{2^{n+1}-2} \mu(k,A) \right\}_{n \ge 0} \right)\\
&= \frac1{\log 2} \gamma \left( \left\{ \frac1{n+1} \sum_{k=0}^{2^{n+1}-2} \mu(k,A) \right\}_{n \ge 0} \right).
\end{aligned}
\end{equation}

Let $0\le A,B \in \mathcal L_{1,\infty}$ be such that 
$$\sum_{k=0}^n \mu(k,A) \le \sum_{k=0}^n \mu(k,B)$$ 
for every $n\ge 0$.
Due to the positivity of $\gamma$, formula~\eqref{eq56} yields $\tau(A) \le \tau(B)$.
Therefore $\tau$ is a fully symmetric functional on $\mathcal L_{1,\infty}$. 
The assertion follows from Theorem~\ref{fs_cor1}.
%
\end{proof}

Next we characterise the Connes-Dixmier traces and the class $\mathcal D_M$, in terms of Banach limits of special types.
To this end we need some preparations.

Recall that $\pi$ denotes the isometric embedding $\pi$ of $l_{\infty}$ into $L_\infty$ given by the formula~\eqref{pi}.
The following result is straightforward. 

The following two lemmas characterise the relation between the extended limits on a sequence space $l_\infty$
and the extended limits on a function space $L_\infty$.

\begin{lem}\label{el1}
(i) For every extended limit $L$ on $L_\infty$ the functional $l$ defined by the formula
 $$l(x) = L(\pi(x)), \ x \in l_\infty$$
 is an extended limit on  $l_\infty$.
 
(ii) For every extended limit $l$ on  $l_\infty$ there exists an extended limit $L$ on $L_\infty$ such that 
 $$l(x) = L(\pi(x)), \ x \in l_\infty.$$
\end{lem}

Define the Hardy (or, integral Ces\`{a}ro) operator $H : L_\infty \to L_\infty$ by the following formula
$$(Hx)(t):=\frac1t \int_0^t x(s)\,ds, \ t>0.$$
Recall that $C_0$ denotes the space of all continuous functions 
from $L_\infty$ vanishing at infinity.

\begin{lem}\label{inv}
(i) For every extended limit $\gamma$ on $L_\infty$ such that $\gamma= \gamma \circ H$ the functional $B$ defined by the formula
\begin{equation}\label{eq1}
B(x) = \gamma(\pi(x)), \ x \in l_\infty 
\end{equation}
 is an extended limit on $l_\infty$ satisfying $B=B\circ C$.
 
(ii) For every extended limit $B$ on $l_\infty$ such that $B=B\circ C$ there exists an extended limit $\gamma$ on $L_\infty$ such that $\gamma= \gamma \circ H$ and
 $$B(x) = \gamma(\pi(x)), \ x \in l_\infty.$$
\end{lem}

\begin{proof}
First of all, for every $x \in l_\infty$ and $t \in (n, n+1]$, $n\ge0$ we have
\begin{align*}
 (H\pi(x))(t) &= \frac1t \left( \int_0^{n+1} \pi(x)(s) \ ds - \int_t^{n+1} \pi(x)(s) \ ds \right)\\
 &= \frac1t \left( \int_0^{n+1} \sum_{k=0}^\infty x_k \chi_{(k,k+1]}(s) \ ds - x_n(n+1-t) \right)\\
 &= \frac1t \left( \sum_{k=0}^{n} x_k  + O(1) \right)
 = \frac1{n+1} \sum_{k=0}^{n} x_k +o(1)
 = (Cx)_n  +o(1).
\end{align*}
Hence, 
\begin{equation}\label{eq2}
H\pi(x) =\pi(Cx)  +o(1). 
\end{equation}

(i) Let $\gamma \in L_\infty^*$ be such an extended limit that $\gamma= \gamma \circ H$. Using~\eqref{eq1} and~\eqref{eq2} we obtain
$$B(Cx) = \gamma(\pi(Cx))=\gamma(H\pi(x))=\gamma(\pi(x))=B(x).$$

Hence, the functional $B \in \mathfrak B$ is such that $B=B\circ C$.

(ii) Let $B$ be an extended limit on $l_\infty$ such that $B=B\circ C$ (in fact, $B \in \mathfrak B$). We set $E:= \pi(l_\infty)$ and
define $\gamma$ on the subspace $E+C_0$ of $L_\infty$ by setting
\begin{equation}\label{eq78}
 \gamma(\pi(x) +\alpha) = B(x)
\end{equation}
for every $x \in l_\infty$ and $\alpha \in C_0$.
It follows from the linearity of $B$ that $\gamma$ is linear on $E+C_0$.
Moreover, for every $x \in l_\infty$ and $\alpha \in C_0$ we obtain
$$\gamma(H(\pi(x)+\alpha)) \stackrel{\eqref{eq2}}{=} \gamma(\pi(Cx)+o(1)) \stackrel{\eqref{eq78}}{=} B(Cx)=B(x)= \gamma(\pi(x) +\alpha).$$
Hence, $\gamma$ is an $H$-invariant linear functional on $E+C_0$.

By the invariant form of the Hahn-Banach theorem~\cite[Theorem 3.3.1]{Edwards} the functional $\gamma$ extends to an $H$-invariant linear functional on $L_\infty$.
Due to construction, $\gamma$ vanishes on $C_0$, that is $\gamma$ is an extended limit on $L_\infty$ and
$$B(x) = \gamma(\pi(x)), \ x \in l_\infty.$$
\end{proof}

\begin{rem}
 The existence of Ces\`aro invariant Banach limits is proved in~\cite{Eberlein}.
 For the extensive study of Ces\`aro invariant Banach limits we refer to~\cite{SS_JFA}.
 For further information on Banach limits with additional invariance properties see~\cite{DPSSS2,SS_JFA,SSU}.
\end{rem}

Define the logarithmic Hardy operator $M$ by the following formula
$$(Mx)(t):=\frac1{\log t} \int_1^t x(s)\,\frac{ds}s, \ x\in L_\infty.$$

\begin{dfn}\label{CD_dfn}
 A trace $\tau$ on $\mathcal L_{1,\infty}$ is said to be a Connes-Dixmier trace, if there exists 
 an extended limit $\gamma$ on $L_\infty$ such that 
 $$\tau(A)= {\rm Tr}_{\gamma \circ M}(A)=(\gamma \circ M)\left( t \mapsto \frac1{\log(1+t)} \int_0^t \mu(s,A)\ ds\right), \ 0 \le A \in \mathcal L_{1,\infty}.$$
\end{dfn}

For technical purposes we introduce the semigroup $P_a, \ a>0$, acting by the formula
$$(P_ax)(t)=x(t^a), \ a>0, \ x\in L_\infty,$$
which is related to the dilation operator $\sigma$ as follows (see~\cite[Proposition 1.3]{CPS2})
\begin{equation}\label{Ps}
 \log\circ \sigma_a - P_a\circ \log \ : \ L_\infty \to C_0, \ a>0.
\end{equation}

Recall that Corollary~\ref{BL1} establishes the linear bijection between the set $\mathcal {PT}$ of all positive normalised traces on $\mathcal L_{1,\infty}$
and the set $\mathfrak B$ of all Banach limits.
The following theorem characterises the class of Connes-Dixmier traces on $\mathcal L_{1,\infty}$ stating
the correspondence between the set $\mathcal C$ and a proper subset of factorisable Banach limits.
\begin{thm}\label{CD}
A trace $\tau$ on $\mathcal L_{1,\infty}$ is a Connes-Dixmier trace if and only if the corresponding  
 Banach limit $B$ (given by Corollary~\ref{BL1}) is of the form $B=\theta \circ C^2$ for some extended limit $\theta$ on $l_\infty$.
\end{thm}

\begin{proof}
 Let $\tau$ be a Connes-Dixmier trace on $\mathcal L_{1,\infty}$ and let $B$ be its corresponding Banach limit given by Corollary~\ref{BL1}.
 By Theorem~\ref{one-to-one1} and Definition~\ref{CD_dfn} we have
 $$B(x) = (\tau \circ {\rm diag})(D x) = 
 (\gamma \circ M)\left( t \mapsto \frac1{\log(1+t)} \int_0^t \pi(Dx)(s) \ ds\right), \ x \in l_\infty.$$
 
A direct verification shows that for every $x \in l_\infty$ and every $t>0$, we have
\begin{equation}\label{dv1}
\int_0^t \pi(Dx)(s) \ ds = \log 2 \cdot \int_0^{\log_2 t} \pi(x)(s) \ ds + O(1). 
\end{equation}

Hence, for every $x \in l_\infty$ the following chain of equalities holds
\begin{align*}
B(x) &=  (\gamma \circ M)\left( t \mapsto \frac{\log 2}{\log t} \int_0^{\log_2 t} \pi(x)(s) \ ds\right) \\
&=  (\gamma \circ M)\left(t \mapsto \frac1{\log t^{\frac1{\log 2}}} \int_0^{\log t^{\frac1{\log 2}}} \pi(x)(s) \ ds\right) \\
&=  (\gamma \circ M\circ P_{\frac1{\log 2}}\circ \log )\left( t \mapsto \frac1t \int_0^t \pi(x)(s) \ ds\right)\\
&= (\gamma \circ M \circ P_{\frac1{\log 2}}\circ \log\circ H)(\pi(x)).
\end{align*}

Since by~\cite[Proposition 1.3]{CPS2} the operators $M$ and $P_a$, $a>0$ commute and 
the operator $\log\circ H - M\circ \log$ maps $L_\infty$ to $C_0$, it follows that
\begin{align*}
\gamma \circ M\circ P_{\frac1{\log 2}} \circ \log \circ H
&= \gamma \circ P_{\frac1{\log 2}} \circ M\circ \log\circ H 
= \gamma \circ P_{\frac1{\log 2}} \circ \log\circ H^2. \\
\end{align*}

Therefore, for every $x \in l_\infty$ we obtain
\begin{equation}\label{123}
B(x)= (\gamma \circ P_{\frac1{\log 2}} \circ \log)(H^2\pi(x))=(\gamma \circ P_{\frac1{\log 2}} \circ \log)(\pi(C^2x)), 
\end{equation}
where the second equality is due to~\eqref{eq2}.

Setting 
$$\theta(y) := (\gamma \circ P_{\frac1{\log 2}} \circ \log)(\pi(y)), \ y\in l_\infty.$$ 
By Lemma~\ref{el1} we see that $\theta$ is an extended limit on $l_\infty$.
By~\eqref{123}, we have $B(x) = \theta(C^2x)$ for every $x\in l_\infty$ and the first assertion is proved.

Suppose now that $B=\theta \circ C^2$ for some extended limit $\theta$ on $l_\infty$. 
By Theorem~\ref{one-to-one1} for every positive $A \in \mathcal L_{1,\infty}$ we obtain
\begin{align*}
 \tau(A)&= B\left( \frac1{\log 2} \cdot \left\{ \sum_{k=2^n-1}^{2^{n+1}-2} \mu(k,A) \right\}_{n \ge 0} \right) \\
 &= \frac1{\log 2} \cdot (\theta \circ C^2) \left(\left\{ \sum_{k=2^n-1}^{2^{n+1}-2} \mu(k,A) \right\}_{n \ge 0} \right) \\
 &= \frac1{\log 2} \cdot (\theta \circ C)\left(\left\{ \frac1n \sum_{k=0}^{2^n} \mu(k,A) \right\}_{n \ge 0} \right).
\end{align*}
By Lemma~\ref{inv}(ii) there exists an extended limit $\gamma_1$ on $L_\infty$ such that \\$\gamma_1(\pi(y))=\theta(y)$ for every $y\in l_\infty$.
Hence,
$$\tau(A)=\frac1{\log 2} \cdot \gamma_1\left(\pi\left(C\left\{ \frac1n \sum_{k=0}^{2^n} \mu(k,A) \right\}_{n \ge 0}\right) \right), \ 0\le A \in \mathcal L_{1,\infty}.$$

By~\eqref{eq2} for every $y \in l_\infty$ we have that $H\pi(y) =\pi(Cy)  +o(1)$. Therefore,
$$\tau(A)=\frac1{\log 2} \cdot (\gamma_1 \circ H)\left(\pi\left(\left\{ \frac1n \sum_{k=0}^{2^n} \mu(k,A) \right\}_{n \ge 0}\right) \right), \ 0\le A \in \mathcal L_{1,\infty}.$$

A direct verification shows that 
\begin{equation}\label{dv2}
\pi\left(\left\{ \frac1n \sum_{k=0}^{2^n} \mu(k,A) \right\}_{n \ge 0}\right) = \left( t \mapsto \frac1t \int_0^{2^t} \mu(s,A) \ ds \right) +o(1), \ 0\le A \in \mathcal L_{1,\infty}. 
\end{equation}

Hence, for every positive $A \in \mathcal L_{1,\infty}$ we obtain
\begin{align*}
\tau(A)&=\frac1{\log 2} \cdot (\gamma_1 \circ H)\left(t \mapsto \frac1t \int_0^{2^t} \mu(s,A) \ ds \right)\\
&=\frac1{\log 2} \cdot (\gamma_1 \circ H\circ \exp) \left(t \mapsto \frac1{\log t} \int_0^{2^{\log t}} \mu(s,A) \ ds \right)\\
&=(\gamma_1 \circ H\circ \exp) \left(t \mapsto \frac1{\log t^{\log 2}} \int_0^{t^{\log 2}} \mu(s,A) \ ds \right)\\
&=(\gamma_1 \circ H\circ \exp \circ P_{\log 2}) \left(t \mapsto \frac1{\log t} \int_0^t \mu(s,A) \ ds \right).
\end{align*}

To finish the proof of this theorem, it suffices to show that there exists an extended limit $\gamma$ on $L_\infty$ such that
$\tau = {\rm Tr}_{\gamma \circ M}$ on the positive cone of $\mathcal L_{1,\infty}$. To this end, we shall show that 
$$\gamma_1 \circ H\circ \exp \circ P_{\log 2} = \gamma \circ M$$
for some extended limit $\gamma$ on $L_\infty$.

Indeed, using~\eqref{Ps} and since $H$ and $\sigma$ commute, it follows that
$$\gamma_1 \circ H\circ \exp \circ P_{\log 2} = \gamma_1 \circ \sigma_{\log 2}\circ H\circ \exp.$$

Since the operator $\exp\circ M - H\circ \exp$ maps $L_\infty$ to $C_0$ (see~\cite[Proposition 1.3]{CPS2}),
if follows that
$$\gamma_1 \circ \sigma_{\log 2}\circ H\circ \exp = \gamma_1 \circ \sigma_{\log 2}\circ \exp \circ M.$$

Setting $\gamma =\gamma_1\circ \sigma_{\log 2}\circ \exp$, we see that $\gamma$ is an extended limit on $L_\infty$ and we obtain that
\begin{align*}
\tau(A) &= (\gamma \circ M) \left(t \mapsto \frac1{\log t} \int_0^t \mu(s,A) \ ds \right)
={\rm Tr}_{\gamma\circ M}(A),
\end{align*}
that is $\tau$ is a Connes-Dixmier trace on $\mathcal L_{1,\infty}$.
\end{proof}

The following subclass of Dixmier traces has been studied in many papers, including~\cite{BF}-\cite{CS} (see also~\cite{SUZ1}).

\begin{dfn}
 A Dixmier trace $\tau$ on $\mathcal L_{1,\infty}$ is said to be generated by a $M$-invariant extended limit (that is, $\tau \in \mathcal D_M$), if there exists 
 an extended limit $\omega$ on $L_\infty$ such that $\omega=\omega\circ M$ and
 $$\tau(A)= {\rm Tr}_\omega(A)=\omega\left( t \mapsto \frac1{\log(1+t)} \int_0^t \mu(s,A)\ ds\right), \ 0 \le A \in \mathcal L_{1,\infty}.$$
\end{dfn}

The following theorem shows that the subclass $\mathcal D_M$ of Dixmier traces, generated by $M$-invariant extended limits, 
corresponds to the set of Ces\`aro invariant Banach limits.

\begin{thm}\label{M-inv}
A Dixmier trace $\tau$ on $\mathcal L_{1,\infty}$ is generated by an $M$-invariant extended limit if and only if the corresponding 
 Banach limit $B$ (given by Corollary~\ref{BL1}) is Ces\`aro invariant, that is $B=B\circ C$.
\end{thm}

\begin{proof}
 Let $\tau \in \mathcal D_M$,  that is for every positive $A \in \mathcal L_{1,\infty}$ we have
 $$\tau(A)= \omega\left( t \mapsto \frac1{\log(1+t)} \int_0^t \mu(s,A)\ ds\right),$$
 for some extended limit $\omega$ on $L_\infty$ satisfying $\omega=\omega \circ M$.
 By Corollary~\ref{BL1} the functional $B = \tau \circ {\rm diag}\circ D$ is a Banach limit. We shall show that $B$ is Ces\`aro invariant.
 We have
 $$B(x) = (\tau \circ {\rm diag})(D x)= \omega\left( t \mapsto \frac1{\log(1+t)} \int_0^t \pi(Dx)(s) \ ds\right), \ x \in l_\infty.$$
 
Using~\eqref{dv1} we obtain
\begin{align*}
B(x) &=  \omega\left( t \mapsto \frac{\log 2}{\log t} \int_0^{\log_2 t} \pi(x)(s) \ ds\right) \\
&=  \omega\left( t \mapsto \frac1{\log t^{\frac1{\log 2}}} \int_0^{\log t^{\frac1{\log 2}}} \pi(x)(s) \ ds\right) \\
&= (\omega \circ P_{\frac1{\log 2}}\circ \log \circ H)(\pi(x)).
\end{align*}

Again by~\cite[Proposition 1.3]{CPS2} we see that the operator $\log\circ H - M\circ \log$ maps $L_\infty$ to $C_0$.
Since the extended limit $\omega$ is $M$-invariant and operators $M$ and $P_a$ commute, it follows that
$$\omega\circ P_{\frac1{\log 2}} \circ \log \circ H 
= \omega\circ P_{\frac1{\log 2}} \circ M\circ \log 
=\omega \circ P_{\frac1{\log 2}} \circ \log,$$
that is the extended limit $\omega\circ P_{\frac1{\log 2}} \circ \log$ on $L_\infty$ is $H$-invariant.

Hence,
$B(x) = (\omega \circ P_{\frac1{\log 2}} \circ \log)(\pi(x))$, $x\in l_\infty$ and
by Lemma~\ref{inv}(i) we see that $B$ is a Ces\`aro invariant Banach limit (that is $B=B\circ C$ on $l_\infty$).
The ``if'' part of the theorem has proved.

Let now $B$ be an extended limit on $l_\infty$ such that $B=B\circ C$. 
We shall show that the weight
$$\tau(A)= B\left(\frac1{\log 2} \cdot \left\{ \sum_{k=2^n-1}^{2^{n+1}-2} \mu(k,A) \right\}_{n \ge 0} \right)$$
defined on the positive cone of $\mathcal L_{1,\infty}$ extends to an element of $\mathcal D_M$.

Since 
\begin{align*}
C \left\{ \sum_{k=2^n-1}^{2^{n+1}-2} \mu(k,A) \right\}_{n \ge 0}
&= \left\{ \frac1{n} \sum_{i=0}^n \sum_{k=2^i-1}^{2^{i+1}-2} \mu(k,A) \right\}_{n \ge 0}\\
&= \left\{ \frac1{n} \sum_{i=0}^{2^{n+1}-2} \mu(k,A) \right\}_{n \ge 0},
\end{align*}
using the Ces\`aro invariance of $B$ we obtain
\begin{align*}
 \tau(A) &= \frac1{\log 2} \cdot (B\circ C) \left(\left\{ \sum_{k=2^n-1}^{2^{n+1}-2} \mu(k,A) \right\}_{n \ge 0} \right) \\
 &= \frac1{\log 2} \cdot B\left(\left\{ \frac1{n} \sum_{k=0}^{2^n} \mu(k,A) \right\}_{n \ge 0} \right), \ 0\le A \in \mathcal L_{1,\infty}.
\end{align*}
By Lemma~\ref{inv}(i) there exists an $H$-invariant extended limit $\gamma$ on $L_\infty$ such that $\gamma(\pi(x))=B(x)$ for every $x\in l_\infty$.
Hence,
$$\tau(A)= \frac1{\log 2} \cdot \gamma\left(\pi\left(\left\{ \frac1{n} \sum_{k=0}^{2^n} \mu(k,A) \right\}_{n \ge 0}\right) \right), 0\le A \in \mathcal L_{1,\infty}.$$

Using~\eqref{dv2} we obtain
\begin{align*}
\tau(A)&=\frac1{\log 2} \cdot \gamma\left(t \mapsto \frac1t \int_0^{2^t} \mu(s,A) \ ds \right)\\
&=\frac1{\log 2} \cdot (\gamma \circ \exp) \left(t \mapsto \frac1{\log t} \int_0^{2^{\log t}} \mu(s,A) \ ds \right)\\
&=(\gamma \circ \exp) \left(t \mapsto \frac1{\log t^{\log 2}} \int_0^{t^{\log 2}} \mu(s,A) \ ds \right)\\
&=(\gamma \circ \exp \circ P_{\log 2}) \left(t \mapsto \frac1{\log t} \int_0^t \mu(s,A) \ ds \right)\\
\end{align*}

Next, we shall prove that the functional $\omega:=\gamma\circ \exp \circ P_{\log 2}$ is $M$-invariant.
Again, we will use the facts (proved in~\cite[Proposition 1.3]{CPS2}) that the operator $\exp\circ M - H\circ \exp$ maps $L_\infty$ to $C_0$
and that the operators $M$ and $P_a$, $a>0$ commute.
Since $\gamma$ is an $H$-invariant extended limit, it follows that
$$\omega\circ M = \gamma\circ \exp \circ P_{\log 2} \circ M = \gamma\circ H\circ \exp \circ P_{\log 2}= \gamma\circ \exp \circ P_{\log 2}= \omega,$$
that is $\omega$ is $M$-invariant.

Since
$$\tau(A)=(\gamma \circ \exp \circ P_{\log 2}) \left(t \mapsto \frac1{\log t} \int_0^t \mu(s,A) \ ds \right) = {\rm Tr}_\omega (A),$$
we conclude that $\tau$ belongs to $\mathcal D_M$.
\end{proof}

\section{Lidskii Formula}\label{lid}

In the present section we first prove the Lidskii formula for self-adjoint operators $A\in \mathcal L_{1,\infty}$, 
then, using Ringrose's representation~\cite[Theorems 1,6,7]{R} of compact operators, we extend the formula to an arbitrary $A\in \mathcal L_{1,\infty}$.

The following elementary lemma will be frequently used in this and subsequent sections.

\begin{lem}\label{aconv}
 For every $x \in l_\infty$ such that $\sum_{k=0}^n x_k =O(1)$, there exists a sequence $y \in l_\infty$ such that
$$\left\{ \sum_{k=2^n-1}^{2^{n+1}-2} x_k \right\}_{n \ge 0}= \{y - Sy\}_{n \ge 0}.$$
In particular, every translation invariant functional on $l_\infty$ vanishes on the sequence $\left\{ \sum_{k=2^n-1}^{2^{n+1}-2} x_k \right\}_{n \ge 0}$.
\end{lem}

\begin{proof}
Setting $y_n= \sum_{k=0}^{2^{n}-2} x_k$, we have
$$
\sum_{k=2^n-1}^{2^{n+1}-2} x_k  =\sum_{k=0}^{2^{n+1}-2} x_k-\sum_{k=0}^{2^{n}-2} x_k= y_{n+1}-y_n = (y-Sy)_{n+1}.
$$
%
\end{proof}

The following theorem is a Lidskii formula for traces on $\mathcal L_{1,\infty}$ and for self-adjoint operators.

\begin{thm}\label{self-adjoint}
 Let $A=A^* \in \mathcal L_{1,\infty}$. For every trace $\tau$ on $\mathcal L_{1,\infty}$
 with the corresponding (by Theorem~\ref{one-to-one1}) $S$-invariant functional $\theta$ on $l_\infty$ the following identity holds
 $$\tau(A) = \theta \left( \frac1{\log 2} \left\{ \sum_{k=2^n-1}^{2^{n+1}-2} \lambda(k,A) \right\}_{n \ge 0} \right).$$
\end{thm}

\begin{proof}
 By Theorem~\ref{one-to-one1} and the linearity of the trace we have
\begin{equation}\label{t1}
\tau(A) = \tau(A_+) -\tau(A_-)
=\theta \left( \frac1{\log 2} \left\{ \sum_{k=2^n-1}^{2^{n+1}-2} \mu(k,A_+)-\mu(k,A_-) \right\}_{n \ge 0} \right). 
\end{equation}

By~\cite[Lemma 5.2.7]{LSZ} for every compact self-adjoint operator $A$ the following estimate holds
$$ \left| \sum_{k=0}^n \lambda(k,A)-\mu(k,A_+)+\mu(k,A_-) \right| \le 2 (n+1) \mu(n,A), \ n\ge0.$$
Hence, if $A=A^* \in \mathcal L_{1,\infty}$, then the right-hand side of the latter inequality is majorized by $2 \|A\|_{\mathcal L_{1,\infty}}$.
So, by Lemma~\ref{aconv} every $S$-invariant functional on $l_\infty$ equals zero on the sequence
$$\left\{ \sum_{k=2^n-1}^{2^{n+1}-2} \lambda(k,A)-\mu(k,A_+)+\mu(k,A_-) \right\}_{n \ge 0}$$
Therefore, 
$$\theta \left(\left\{ \sum_{k=2^n-1}^{2^{n+1}-2} \lambda(k,A) \right\}_{n \ge 0} \right)=
\theta \left(\left\{ \sum_{k=2^n-1}^{2^{n+1}-2} \mu(k,A_+)-\mu(k,A_-) \right\}_{n \ge 0} \right)$$
for every $S$-invariant functional $\theta$ on $l_\infty$.
Combining the latter equality with~\eqref{t1} we obtain the required assertion.
\end{proof}

The following theorem is a Lidskii formula for all traces on the ideal $\mathcal L_{1,\infty}$.
This result extends and complements the corresponding results from~\cite{AS,CPS2,CRSS,SSZ,SUZ3}.

\begin{thm}\label{Lidskii}
For every $A \in \mathcal L_{1,\infty}$ and every trace $\tau$ on $\mathcal L_{1,\infty}$
 with the corresponding (by Theorem~\ref{one-to-one1}) $S$-invariant functional $\theta$ on $l_\infty$ the following identity holds
 $$\tau(A) = \theta \left( \frac1{\log 2} \left\{ \sum_{k=2^n-1}^{2^{n+1}-2} \lambda(k,A) \right\}_{n \ge 0} \right).$$
\end{thm}

\begin{proof}
For every compact operator $A$ there exist a compact normal operator $N$ 
and a compact quasi-nilpotent operator $Q$ such that $A=N+Q$ and $\lambda(A)=\lambda(N)$ ~\cite[Theorems 1,6,7]{R} (in particular, $\mu(N)=|\lambda(N)|=|\lambda(A)|$).
By the Weyl theorem (see e.g.~\cite[Theorem 3.1]{GK}), the sequence $|\lambda(A)|$ is logarithmically majorized by the sequence $\mu(A)$. Recall that (see Proposition 3.2 in \cite{K}) the quasi-norm in $\mathcal L_{1,\infty}$ is monotone with respect to the logarithmic majorization. Thus, $\|\lambda(A)\|_{1,\infty}\leq{\rm const}\cdot\|A\|_{1,\infty}.$ Since $\mu(N)=|\lambda(A)|,$ it follows that $N\in \mathcal L_{1,\infty}$ and, therefore, $Q\in \mathcal L_{1,\infty}.$ By~\cite[Theorem 5.5.1]{LSZ} (see also~\cite{K}), we have $\tau(Q)=0$ for every quasi-nilpotent operator $Q$ and for every trace on $\mathcal L_{1,\infty}$.

Hence,
$$\tau(A)=\tau(N)=\tau(\Re(N))+i\tau(\Im(N)),$$
(where $\Re(N)$ and $\Im(N)$ are real and imaginary parts of the operator $N$, respectively)
and by Theorem~\ref{self-adjoint} we obtain
\begin{equation}\label{t2}
\tau(A) 
=\theta \left( \frac1{\log 2} \left\{ \sum_{k=2^n-1}^{2^{n+1}-2} \lambda(k,\Re(N))+i\lambda(k,\Im(N)) \right\}_{n \ge 0} \right). 
\end{equation}

By~\cite[Lemma 5.2.10]{LSZ} for every compact normal operator $N$ the following estimate holds
$$ \left| \sum_{k=0}^n \lambda(k,N)-\lambda(k,\Re(N))-i\lambda(k,\Im(N)) \right| \le 5 n \mu(n,N).$$

Hence, for $N \in \mathcal L_{1,\infty}$ the right-hand side is majorized by $5 \|N\|_{\mathcal L_{1,\infty}}$.
Lemma~\ref{aconv} now yields that every $S$-invariant functional on $l_\infty$ equals zero on the sequence
$$\left\{ \sum_{k=2^n-1}^{2^{n+1}-2} \lambda(k,N)-\lambda(k,\Re(N))-i\lambda(k,\Im(N)) \right\}_{n \ge 0}.$$
Therefore, for every $S$-invariant functional $\theta$ on $l_\infty$ we have
$$\theta \left(\left\{ \sum_{k=2^n-1}^{2^{n+1}-2} \lambda(k,N) \right\}_{n \ge 0} \right)=
\theta \left(\left\{ \sum_{k=2^n-1}^{2^{n+1}-2} \lambda(k,\Re(N))+i\lambda(k,\Im(N)) \right\}_{n \ge 0} \right).$$

Combining the latter equality with~\eqref{t2} we obtain the claim.

\end{proof}


\section{Measurability}\label{Me}

Using the results of Corollary~\ref{BL1} together with those of Theorems~\ref{FS1}, \ref{CD}, \ref{M-inv} and the Lidskii formula
from the preceeding section (Theorem~\ref{Lidskii}), we can easily infer criteria for measurability of operators within $\mathcal L_{1,\infty}$ with respect to various subclasses 
of normalised traces on $\mathcal L_{1,\infty}$. We recall the following definition from~\cite{C,LS}.

\begin{dfn}
Let $\mathcal A$ be a subset of all traces on $\mathcal L_{1,\infty}$.
An operator $A \in \mathcal L_{1,\infty}$ is called $\mathcal A$-measurable if
the values of all traces from $\mathcal A$ coincide on $A$.
\end{dfn}

Propositions~\ref{K_S_meas},~\ref{D_meas},~\ref{CD_meas},~\ref{D_M_meas} provide definitive results, in terms of eigenvalue sequences, concerning 
measurability with respect to the classes of all positive normalised traces ($\mathcal {PT}$), 
all Dixmier traces ($\mathcal D$), all Connes-Dixmier traces ($\mathcal C$) 
and all Dixmier traces generated by $M$-invariant extended limits ($\mathcal D_M$).
For the ideal $\mathcal L_{1,\infty}$ these results strengthen and complete corresponding results from~\cite{LSS,SS_JFA,SUZ1,SUZ2}.

The following theorem resolves (in the class of positive normalized traces) an open problem discussed in~\cite[p. 1061]{CS}. 
In fact, it appears that the class $\mathcal {PT}$ is the largest class of traces for which the meaningful description 
of the corresponding measurable elements is possible.
\begin{prop}\label{K_S_meas}
 An operator $A \in \mathcal L_{1,\infty}$ is  $\mathcal {PT}$-measurable if and only if 
 the sequence 
 $$\left\{ \sum_{k=2^n-1}^{2^{n+1}-2} \lambda(k,A) \right\}_{n \ge 0}$$ is almost convergent.
Here $\{ \lambda(n,A) \}_{n \ge 0}$ is any eigenvalue sequence of $A$.
\end{prop}

\begin{proof}
 An operator $A \in \mathcal L_{1,\infty}$ is  $\mathcal {PT}$-measurable if and only if 
 $\tau(A)=a$ for every positive normalised trace $\tau$ on $\mathcal L_{1,\infty}$. 
 By Corollary~\ref{BL1} and Theorem~\ref{Lidskii}, the previous statement is equivalent to the fact that 
 $$B\left( \frac1{\log 2} \left\{ \sum_{k=2^n-1}^{2^{n+1}-2} \lambda(k,A) \right\}_{n \ge 0} \right) =a$$
 for every Banach limit $B$. The assertion follows now from Definition~\ref{ac}.
\end{proof}

It is shown in Corollary~\ref{fs_cor1} that the classes of Dixmier traces (on $\mathcal L_{1,\infty}$) 
and normalised fully symmetric functionals on $\mathcal L_{1,\infty}$ coincide.
Hence, the following theorem also resolves (for the class $\mathcal L_{1,\infty}$) an open problem (iii) stated~\cite[p. 1061]{CS},
concerning the measurability with respect to the class of all normalised fully symmetric functionals.

\begin{prop}\label{D_meas}
 An operator $A \in \mathcal L_{1,\infty}$ is  Dixmier-measurable if and only if 
 the sequence 
 $$C\left\{ \sum_{k=2^n-1}^{2^{n+1}-2} \lambda(k,A) \right\}_{n \ge 0}$$ is convergent.
Here $\{ \lambda(n,A) \}_{n \ge 0}$ is any eigenvalue sequence of $A$.
\end{prop}

\begin{proof}
 An operator $A \in \mathcal L_{1,\infty}$ is  $\mathcal D$-measurable if and only if 
 $\tau(A)=a$ for every Dixmier trace $\tau$ on $\mathcal L_{1,\infty}$. 
 By Theorems~\ref{FS1} and~\ref{Lidskii}, the previous statement is equivalent to the fact that 
 $$(\gamma \circ C)\left( \frac1{\log 2} \left\{ \sum_{k=2^n-1}^{2^{n+1}-2} \lambda(k,A) \right\}_{n \ge 0} \right) =a$$
 for every extended limit $\gamma$. The assertion follows from the Remark~\ref{extended_limits}.
\end{proof}

The following result shows that the concepts of Dixmier and $\mathcal {PT}$-measura\-bility 
differ even on the positive cone of $\mathcal L_{1,\infty}$.

\begin{thm}\label{dif1}
 The class of $\mathcal D$-measurable operators is strictly wider than the class of $\mathcal {PT}$-measurable operators.
\end{thm}

\begin{proof}
 Consider the sequence 
 $$y=\sum_{n=1}^\infty \chi_{[2^n,2^n+n]} +\emm \in l_\infty.$$
 
 It is easy to check that $(Cy)_n \mathop{\longrightarrow}\limits_{n\to\infty} 1$ and $y$ is not almost convergent.
Since $y_n \ge \frac{y_{n+1}}2$ for every $n\ge0$, it follows that $\frac{y_n}{2^n}\ge \frac{y_{n+1}}{2^{n+1}}$ and $(Dy)^* = Dy$.
 
 For $A = {\rm diag} (Dy) \in \mathcal L_{1,\infty}$ we clearly have that $\lambda(A) = (Dy)$.
 
Using the definition of the operator $D$, we obtain
 $$\sum_{k=2^n-1}^{2^{n+1}-2} \lambda(k,A) = \sum_{k=2^n-1}^{2^{n+1}-2} (Dy)_k = y_n.$$
 
 By Propositions~\ref{K_S_meas} and~\ref{D_meas} we obtain that the operator $A$ is $\mathcal D$-measurable, 
 but $A$ is not $\mathcal {PT}$-measurable.
\end{proof}

The following proposition characterises Connes-Dixmier measurability of an operator $A \in \mathcal L_{1,\infty}$ 
in terms of its eigenvalue sequence.

\begin{prop}\label{CD_meas}
 An operator $A \in \mathcal L_{1,\infty}$ is Connes-Dixmier measurable if and only if 
 the sequence 
 $$C^2\left\{ \sum_{k=2^n-1}^{2^{n+1}-2} \lambda(k,A) \right\}_{n \ge 0}$$ is convergent.
Here $\{ \lambda(n,A) \}_{n \ge 0}$ is any eigenvalue sequence of $A$.
\end{prop}

\begin{proof}
 An operator $A \in \mathcal L_{1,\infty}$ is $\mathcal C$-measurable if and only if 
 $\tau(A)=a$ for every Connes-Dixmier trace $\tau$ on $\mathcal L_{1,\infty}$. 
 By Theorem~\ref{CD} and Theorem~\ref{Lidskii}, the previous statement is equivalent to the fact that 
 $$(\gamma \circ C^2)\left( \frac1{\log 2} \left\{ \sum_{k=2^n-1}^{2^{n+1}-2} \lambda(k,A) \right\}_{n \ge 0} \right) =a$$
 for every extended limit $\gamma$. The assertion follows from Remark~\ref{extended_limits}.
\end{proof}

To prove the main result of this section we need Hardy's Tauberian theorem for Ces\`aro summability (see, e.g.~\cite[Chapter 6.8]{H}).
\begin{thm}\label{taub}
If $x \in l_\infty$ is such that the sequence $\{n(x_n -x_{n-1})\}_{n\ge 1}$ is bounded from below,
then the sequence $Cx$ is convergent if and only if the sequence $x$ is convergent.
\end{thm}

The result of the following theorem complements Theorem 3.7 from~\cite{LSS}. The cited
theorem showed the coincidence of the sets of positive Dixmier- and Connes-Dixmier measurable operators from $\mathcal M_{1,\infty}$. 
On the smaller ideal $\mathcal L_{1,\infty}$ the condition of positivity can be dropped.

The following result resolves in the affirmative the problem (i) stated in~\cite[p. 1061]{CS} (in the ideal $\mathcal L_{1,\infty}$).
\begin{thm}\label{D=CD}
An operator $A \in \mathcal L_{1,\infty}$ is  Connes-Dixmier measurable if and only if it is Dixmier-measurable.
\end{thm}

\begin{proof}
The condition of $\mathcal D$-measurability evidently implies $\mathcal C$-measurability. 

Let an operator $A \in \mathcal L_{1,\infty}$ be $\mathcal C$-measurable, that is by Proposition~\ref{CD_meas}
the sequence 
 $$C^2\left\{ \sum_{k=2^n-1}^{2^{n+1}-2} \lambda(k,A) \right\}_{n \ge 0}$$ is convergent.
 We have to show that the operator $A \in \mathcal L_{1,\infty}$ is $\mathcal D$-measurable, or equivalently by Proposition~\ref{D_meas}, that
the sequence 
 $$C\left\{ \sum_{k=2^n-1}^{2^{n+1}-2} \lambda(k,A) \right\}_{n \ge 0}$$ is convergent.
 
To this end, we need to show that for every $A \in \mathcal L_{1,\infty}$ the sequence 
$$C\left\{ \sum_{k=2^n-1}^{2^{n+1}-2} \lambda(k,A) \right\}_{n \ge 0}$$
satisfies the condition of Theorem~\ref{taub}.

Indeed, a direct verification shows that for every $y \in l_\infty$ the following estimate holds
$$n((Cy)_n -(Cy)_{n-1}) = y_n -(Cy)_n \ge -2 \|y\|_{l_\infty}.$$

\end{proof}

The following result should be compared with~\cite[Corollary 21]{SS_JFA}, describing the set of 
$\mathcal D_M$-measurable operators from $\mathcal M_{1,\infty}$.

\begin{prop}\label{D_M_meas}
 An operator $A \in \mathcal L_{1,\infty}$ is $\mathcal D_M$-measurable if and only if
 $$\lim_{m\to\infty} \liminf_{n\to\infty}C^m\left\{ \sum_{k=2^n-1}^{2^{n+1}-2} \lambda(k,A) \right\}_{n \ge 0} =
 \lim_{m\to\infty} \limsup_{n\to\infty}C^m\left\{ \sum_{k=2^n-1}^{2^{n+1}-2} \lambda(k,A) \right\}_{n \ge 0}.$$
Here $\{ \lambda(n,A) \}_{n \ge 0}$ is any eigenvalue sequence of $A$.
\end{prop}

\begin{proof}
From Theorem~\ref{M-inv} and Theorem~\ref{Lidskii} we obtain that the operator $A \in \mathcal L_{1,\infty}$ is $\mathcal D_M$-measurable
if and only if all Ces\`aro invariant Banach limits (that is $B=B\circ C$) take the same value on the sequence $\left\{ \sum_{k=2^n-1}^{2^{n+1}-2} \lambda(k,A) \right\}_{n \ge 0}$.
By~\cite[Theorem 5, Corollary 13]{SS_JFA} all Ces\`aro invariant Banach limits take the same value on the sequence $x\in l_\infty$ if and only
 $$\lim_{m\to\infty} \liminf_{n\to\infty}(C^m x)_n =  \lim_{m\to\infty} \limsup_{n\to\infty}(C^m x)_n,$$
which proves the assertion.
\end{proof}

The following theorem shows that Dixmier-measurability  
differs from $\mathcal D_M$-measura\-bility even on the positive cone of $\mathcal L_{1,\infty}$.
It improves the corresponding result for the ideal $\mathcal M_{1,\infty}$ from~\cite[Theorem 24]{SS_JFA}.

\begin{thm}\label{dif2}
The class of $\mathcal D$-measurable operators from $\mathcal L_{1,\infty}$ 
is strictly contained in the class of $\mathcal D_M$-measurable operators.
\end{thm}

\begin{proof}
 Consider the sequence 
 $$y=\sum_{n=1}^\infty \chi_{(2^{2n},2^{2n+1}]} +\emm \in l_\infty.$$
 
 It is easy to check that the sequence $Cy$ is not convergent. 
 By~\cite[Theorem 15]{SS_JFA} for the sequence 
 $$x_k=(-1)^n, \ 2^n< k \le 2^{n+1}, \ n \in \mathbb N$$
 we have that $B(x)=0$ for every Ces\`aro invariant Banach limit $B$.
 Since $y \in  x/2 + 3/2 \cdot \emm +c_0$, it follows that 
  $B(y)=3/2$ for every Ces\`aro invariant Banach limit $B$.
  
Since $y_n \ge \frac{y_{n+1}}2$ for every $n\ge0$, it follows that $\frac{y_n}{2^n}\ge \frac{y_{n+1}}{2^{n+1}}$ and $(Dy)^* = Dy$.
 
 For $A = {\rm diag} (Dy) \in \mathcal L_{1,\infty}$ we clearly have that $\lambda(A) = (Dy)$.
Using the definition of the operator $D$, we obtain
 $$\sum_{k=2^n-1}^{2^{n+1}-2} \lambda(k,A) = \sum_{k=2^n-1}^{2^{n+1}-2} (Dy)_k = y_n.$$
 
 By Propositions~\ref{D_M_meas} and~\ref{D_meas} we obtain that the operator $A$ is $\mathcal D_M$-measu\-rable
 and $A$ is not Dixmier-measurable. 
\end{proof}

\section{Application to pseudo-differential operators}\label{NCG}

Connes' trace theorem,~\cite[Theorem 1]{C3}, states that a Dixmier trace applied to a compactly supported classical pseudo-differential operator of order $-d$ 
yields the Wodzicki's residue up to a constant. 
This enables the Dixmier trace of any compactly supported classical pseudo-differential operator of order $-d$ to be calculated from its symbol.

In this section, with the aid of the results established, we provide a version of Connes' trace theorem for positive normalised traces.
Following the ideas of~\cite{KLPS} (see also~\cite{LSZ}) we introduce the class of so-called Laplacian modulated operators 
and the residue mapping ${\rm Res}$, which extends the Wodzicki's residue.

Let us first give a definition of a pseudo-differential operator, see e.g.~\cite[Definition 10.2.6]{LSZ}.

\begin{dfn}
Let $m \in \mathbb R$.
A function $p \in C^\infty(\mathbb R^d,\mathbb R^d)$ satisfying the condition
$$\sup_{x,s} |\partial_x^\alpha  \partial_s^\beta p(x,s)| (1+|s|^2)^\frac{|\beta|-m}2 < \infty$$
for every multi-indices $\alpha, \beta \in (\mathbb N \cup \{0\})^d$ is called
a symbol of order $m$.
\end{dfn}
In general terminology, we have just defined the uniform symbol of \\H\"{o}rmander type $(1,0)$, see e.g.~\cite{Hormander} and~\cite[Chapter 2]{RT}.


By $\mathcal S(\mathbb R^d)$ we denote the space of Schwartz functions (the smooth functions of rapid decay).


\begin{dfn}
Let $m \in \mathbb R$ and let $p$ be a symbol of order $m$.
 The operator $A : \mathcal S(\mathbb R^d) \to \mathcal S(\mathbb R^d)$ given by the formula
 $$(Au)(x) := \int_{\mathbb R^d} \int_{\mathbb R^d} e^{i \langle x-y, s\rangle} p(x,s) u(y) \ dy \ ds, \ u \in \mathcal S(\mathbb R^d)$$
 is called a pseudo-differential operator of order $m$.
\end{dfn}

\begin{dfn}
A pseudo-differential operator $A$ of order $m$ is called classical if its symbol has an asymptotic expansion
$$p \sim \sum_{j=0}^\infty p_{m-j},$$
where each $p_{m-j}:=p_{m-j}(x,s)$ is a symbol of order $m-j$ and is a homogeneous function of order $m-j$ in the variable $s \in \mathbb R^d$ except near zero.
\end{dfn}

Next we introduce a pseudo-differential operator of a particular type.
For a smooth function with compact support $\phi \in C_c^\infty(\mathbb R^d)$ we define the multiplication operator 
$(M_\phi f)(x) = \phi(x) f(x)$, $f \in \mathcal S(\mathbb R^d)$.

\begin{dfn}\label{def:compbase}
A pseudo-differential operator $A : \mathcal S(\mathbb R^d) \to \mathcal S(\mathbb R^d)$ is said to be compactly supported if 
$M_\phi A M_\psi= A$ 
for some $\phi, \psi \in C_c^\infty(\mathbb R^d)$.
\end{dfn}

Pseudo-differential operators $A : \mathcal S(\mathbb R^d) \to \mathcal S(\mathbb R^d)$ associated to the class of symbols of order $m<0$
generally do not extend to a compact linear operators $A : L_2(\mathbb R^d) \to  L_2(\mathbb R^d)$.
However, by~\cite[Theorem 10.2.22]{LSZ} a compactly supported pseudo-differential operator $A$ of order $m<0$ extends to a compact linear operator $A : L_2(\mathbb R^d) \to  L_2(\mathbb R^d)$
and a compactly supported pseudo-differential operator $A$ of order $m<-d$ extends to a trace class operator.

Let $\mathcal L_2$ denote the class of Hilbert-Schmidt operators on the Hilbert space $L_2(\mathbb R^d)$.
%
%
%
%
Let $\Delta=\sum_{i=1}^d \frac{\partial^2}{\partial x_i^2}$ be the Laplacian on $\mathbb R^d$.
The following definitions were introduced in~\cite{KLPS}.
\begin{dfn} \label{def:Lapmod}
Let $d\in \mathbb N$. A bounded operator $A:L_2(\mathbb R^d)\to L_2(\mathbb R^d)$ 
is called Laplacian modulated if 
$$\sup_{t>0}t^{1/2} \|A(1+t(1-\Delta)^{-d/2})^{-1}\|_{\mathcal L_2} < \infty.$$
\end{dfn}

If follows from the definition that every Laplacian modulated operator $A$ is Hilbert-Schmidt,
so it has a unique symbol in $L_2(\mathbb R^d,\mathbb R^d)$ denoted by $p_A$.

By~\cite[Theorem 11.3.17]{LSZ} for every compactly supported pseudo-differential operator $A : \mathcal S(\mathbb R^d) \to \mathcal S(\mathbb R^d)$
of order $-d$ its extension to a compact linear operator $A : L_2(\mathbb R^d) \to  L_2(\mathbb R^d)$ is Laplacian modulated.

According to~\cite[Remark 11.3.14]{LSZ} an operator $A$ on $L_2(\mathbb R^d)$ is Laplacian modulated if and only its symbol $p_A$
satisfies the condition
\begin{equation}\label{mod}
\sup_{t>0} (1+t)^{d/2} \left( \int_{|s|>t} \int_{\mathbb R^d} |p_A(x,s)|^2 \ dx ds\right)^{1/2} < \infty. 
\end{equation}


It was shown in~\cite[Lemma 6.12]{KLPS} that, for every compactly supported Laplacian modulated operator $A$ with symbol $p_A$,
the sequence 
$$\left\{ \frac1{\log(2+n)} \int_{\mathbb R^d} \int_{|s|\le n^{1/d}} p_A(x,s) ds\,dx \right\}_{n\ge0}$$
is bounded. Therefore, the following definition makes sense.

\begin{dfn}\label{def:res}
The linear map
$$
A \mapsto {\rm Res}(A):= \left[ \frac1{\log(2+n)} \int_{\mathbb R^d} \int_{|s|\le n^{1/d}} p_A(x,s) ds\,dx \right]
$$
from the set of all compactly supported Laplacian modulated operators to $\ell_\infty / c_0$
is called the residue, where $[ \cdot ]$ denotes the equivalence class in $\ell_\infty / c_0$.
\end{dfn}
Note that any sequence $\{{\rm Res}_n(A)\}_{n \ge 0} \in  l_\infty$ such that
\begin{equation}\label{eq:res}
\int_{\mathbb R^d} \int_{|s|\le n^{1/d}} p_A(x,s) ds\,dx = \mathrm{Res}_n(A)  \log n + o(\log n)
\end{equation}
defines the residue ${\rm Res}(A) = [{\rm Res}_n(A)] \in \ell_\infty / c_0$.
In this section, by \lq\lq scalars \rq\rq in $\ell_\infty / c_0$ we mean the classes of convergent sequences. That is if $\alpha \in \mathbb C$, then  $\alpha \equiv [ a_n ]$ where $\lim_{n \to \infty} a_n = \alpha$.

The following result shows that the residue ${\rm Res}$ is the extension of Wodzicki's  residue ${\rm Res}_W$, introduced in~\cite{Wodzicki}.
It was proved in~\cite[Proposition 6.16]{KLPS} (see also~\cite[Proposition 11.3.21]{LSZ}).

\begin{prop}\label{prop:wodres}
Let $P$ be a compactly supported classical pseudo-differential operator of order $-d$. We have that 
${\rm Res}(P)$ is the scalar
$$
{\rm Res}(P) =  {\rm Res}_W(P) := \int_{\mathbb R^d} \int_{\mathbb S^{d-1}} p_{-d}(x,s) ds \, dx$$
where ${\rm Res}_W$ denotes the Wodzicki's residue and $p_{-d}$ denotes the principal symbol of $P$.
\end{prop}

The following generalisation of Connes' trace theorem was proved in~\cite[Theorem 6.32]{KLPS}.

\begin{thm}~\label{CTTv1}
Let $A$ be a compactly supported Laplacian modulated operator with symbol $p_A$. We have $A \in \mathcal L_{1,\infty}(L_2(\mathbb R^d))$.  Moreover,
\begin{enumerate}
\item[(i)] for a Dixmier trace $\mathrm{Tr}_\omega$,
$$
\mathrm{Tr}_\omega(A) = \frac1{d(2\pi)^d} \omega({\rm Res}(A))
$$
where ${\rm Res}(A) \in l_\infty / c_0$ is the residue of $A$;
\item[(ii)]
$$
\mathrm{Tr}_\omega(A) = \frac{1}{d(2\pi)^d} {\rm Res}(A)
$$
for every Dixmier trace ${\rm Tr}_\omega$ if and only if the residue ${\rm Res}(A)$ is scalar;
\item[(iii)]
$$
\tau(A)= \frac1{d(2\pi)^d} {\rm Res}(A) 
$$
for every normalised trace $\tau$ on  $\mathcal L_{1,\infty}(L_2(\mathbb R^d))$
if and only if the residue ${\rm Res}(A)$ is a scalar and
\begin{equation}\label{classical}
\int_{\mathbb R^d} \int_{|s|\le n^{1/d}} p_A(x,s) ds\,dx = \frac1d {\rm Res}(A) \, \log n + O(1).
\end{equation}
\end{enumerate}
\end{thm}

The following result proved in~\cite[Theorem 6.23]{KLPS} lies at the heart of Connes' trace theorem.

\begin{thm}\label{eigen}  Suppose $A:L_2(\mathbb R^d)\to L_2(\mathbb R^d)$ is compactly supported and Laplacian modulated.  
We have that $A\in\mathcal L_{1,\infty}(L_2(\mathbb R^d))$ and 
\begin{equation}\label{1000} 
 \sum_{k=0}^n\lambda(k,A)-\frac{1}{(2\pi)^d}\int_{\mathbb R^d}\int_{|s|\le n^{1/d}} p_A(x,s)ds dx = O(1)
\end{equation}
where $\{ \lambda(k,A) \}_{k=1}^\infty$ is an eigenvalue sequence of $A$ and $p_A$ is the symbol of $A$.
\end{thm}

A consequence of Theorem~\ref{CTTv1} part (iii) is that all traces on $\mathcal L_{1,\infty}$ 
applied to a classical pseudo-differential operator yield the same value, 
\cite[Corollary 6.35]{KLPS}. 
One of the reasons for generalising Connes' trace theorem is to understand traces of pseudo-differential operators of order $-d$ 
that are not classical pseudo-differential operators.  
Theorem~\ref{Lidskii} provides an explicit formula for the positive trace of a compactly supported Laplacian modulated operator in terms of eigenvalues, 
and therefore it stands to reason given Theorem~\ref{eigen} that we can obtain a formula for calculating 
\emph{any} positive trace of a pseudo-differential operator of order $-d$ using its symbol. 
The following theorem is the main result of this section, it extends part (i) in Theorem~\ref{CTTv1} above 
and complements parts (ii) and (iii). It should be compared with Theorem 11.5.1 in \cite{LSZ}.

\begin{thm}\label{CTT}
Let $A$ be a compactly supported Laplacian modulated operator with symbol $p_A$. We have $A \in \mathcal L_{1,\infty}(L_2(\mathbb R^d))$.  Moreover,
\begin{enumerate}
\item[(i)] for any normalised positive trace $\tau$,
$$
\tau(A) = \frac1{(2\pi)^d \log 2}  B \left( \left\{ \int_{\mathbb R^d}\int_{2^{n/d}<|s|\le 2^{(n+1)/d}} p_A(x,s)ds dx \right\}_{n \ge 0} \right),
$$
where $B$ is the Banach limit corresponding to $\tau$ (given by Corollary~\ref{BL1});
\item[(ii)] the equality 
$$\tau(A)= \frac1{d(2\pi)^d} {\rm Res}(A)$$
holds for every positive normalised trace $\tau$ on $\mathcal L_{1,\infty}(L_2(\mathbb R^d))$
if and only if the residue ${\rm Res}(A)$ is a scalar and the sequence
\begin{equation}\label{res}
\left\{ \int_{\mathbb R^d} \int_{2^\frac{n}d <|s|\le 2^\frac{n+1}d} p_A(x,s) ds\,dx \right\}_{n\ge0} 
\end{equation}
is almost convergent to the number $\frac1d \log 2 \cdot {\rm Res}(A)$.
\end{enumerate}
\end{thm}

\begin{proof}
By Theorem~\ref{eigen} we have
$$\sum_{k=0}^n \left(\lambda(k,A)-\frac{1}{(2\pi)^d}\int_{\mathbb R^d}\int_{(k-1)^{1/d}<|s|\le k^{1/d}} p_A(x,s)ds dx \right)= O(1).$$

Hence, by Lemma~\ref{aconv} the sequence
$$\left\{\sum_{k=2^n+1}^{2^{n+1}} \left(\lambda(k,A)-\frac{1}{(2\pi)^d}\int_{\mathbb R^d}\int_{(k-1)^{1/d}<|s|\le k^{1/d}} p_A(x,s)ds dx \right)\right\}_{n\ge0}$$
belongs to the space $ac_0$ of almost convergent sequences, or equivalently, the sequence
\begin{equation}\label{t3}
\left\{\sum_{k=2^n-2}^{2^{n+1}-1}\lambda(k,A)-\frac{1}{(2\pi)^d}\int_{\mathbb R^d}\int_{2^{n/d}<|s|\le 2^{(n+1)/d}} p_A(x,s)ds dx \right\}_{n\ge0} \in ac_0. 
\end{equation}
By Corollary~\ref{BL1} and Theorem~\ref{Lidskii} we have that for every positive normalised trace $\tau$ on $\mathcal L_{1,\infty}$ 
there exists a Banach limit $B$ such that
$$\tau(A) = B \left( \frac1{\log 2} \left\{\sum_{k=2^n-2}^{2^{n+1}-1}\lambda(k,A) \right\}_{n\ge0} \right).$$
We have now proved the assertion of (i), since every Banach limit vanishes on $ac_0$ (see~\eqref{t3}).

Now, the equality 
$$\tau(A)= \frac1{d(2\pi)^d} {\rm Res}(A)$$
holds for every positive normalised trace $\tau$ if and only if for every Banach limit $B$ we have
$$B \left( \frac1{(2\pi)^d} \left\{ \int_{\mathbb R^d} \int_{2^\frac{n}d <|s|\le 2^\frac{n+1}d} p_A(x,s) ds\,dx \right\}_{n\ge0} \right) = \frac{\log 2}{d(2\pi)^d} {\rm Res}(A).$$
That is if and only if the sequence
$$ \left\{ \int_{2^\frac{n}d <|s|\le 2^\frac{n+1}d} p_A(x,s) ds\,dx \right\}_{n\ge0}$$
is almost convergent to the number $\frac1d \log 2 \cdot {\rm Res}(A).$
\end{proof}

To show that Theorem~\ref{CTT}(i) is truly an extension of Theorem~\ref{CTTv1}(i)  
we need to show that there are pseudo-differential operators whose value 
for a positive trace cannot be calculated by the formula for a Dixmier trace. 
An example is given by Theorem~\ref{psdo} below. 
Before we state it we need some technical preparations.

\begin{lem}\label{lemma1}
 We have
 $$\int_3^s \sin \left(\frac{t}{\log t}\right)dt =O(\log s),\quad s>3.$$
\end{lem}

\begin{proof}
 Set 
 $$I:= \int_3^s \sin \left(\frac{t}{\log t}\right)dt = \int_3^s \sin \left(\frac{t}{\log t}\right)\frac{\log t-1}{\log^2 t}\frac{\log^2 t}{\log t-1}dt.$$
 
 Integrating by parts we obtain
 
$$I= -\cos \left(\frac{t}{\log t}\right)\frac{\log^2 t}{\log t-1}\Big\vert_3^s + 
\int_3^s \cos \left(\frac{t}{\log t}\right) \cdot \frac1t \frac{\log^2 t - 2 \log t}{(\log t-1)^2} dt.$$

We clearly have
$$\cos \left(\frac{t}{\log t}\right)\frac{\log^2 t}{\log t-1}\Big\vert_3^s =O(\log s)$$ 
and 
$$\int_3^s \cos \left(\frac{t}{\log t}\right) \cdot \frac1t \frac{\log^2 t - 2 \log t}{(\log t-1)^2} dt
= O\left( \int_3^s \frac{dt}t \right) =  O(\log s).$$

Hence, $I= O(\log s)$.
\end{proof}

\begin{lem}\label{lemma2}
The sequence
\begin{equation}\label{e32}
 \left\{ \int_{\log2^{n/d}}^{\log 2^{(n+1)/d}} \sin \frac{z}{\log z} dz \right\}_{n\ge0}
\end{equation}
is not almost convergent to zero.
\end{lem}

\begin{proof}
We shall show that 
\begin{align*}
 \lim_{n\to\infty} \sup_{m\ge0} \frac1n\sum_{k=m}^{m+n-1} \int_{\frac{k}d \log2}^{\frac{k+1}d \log 2} \sin \frac{z}{\log z} dz
 &= \lim_{n\to\infty} \sup_{m\ge0} \frac1n \int_{\frac{m}d \log2}^{\frac{n+m}d \log 2} \sin \frac{z}{\log z} dz\\
 &\ge\frac{\log 2}d,
\end{align*}
which means, in view of Theorem~\ref{lorentz}, that the sequence~\eqref{e32} is not almost convergent to zero.

For every $n \in \mathbb N$ there exists $z_n \in \mathbb R$ such that $\frac{z_n}{\log z_n} = 2^{n^2}\pi +\pi/2$.
Note that
$\frac1{\log z_n}=O(\frac1{n^2})$.

Set $m = \lfloor z_n \frac{d}{\log 2} \rfloor$ (the integral part of $ z_n\frac{d}{\log 2}$). We have
$$J := \int_{\frac{m}d \log2}^{\frac{n+m}d \log 2} \sin \frac{z}{\log z} dz = \int_{z_n}^{z_n+\frac{n}d \log 2} \sin \frac{z}{\log z} dz +O(1).$$

By the Mean Value Theorem, for every $z \in [z_n, z_n+\frac{n}d \log 2]$ we obtain
$$
\sin \frac{z}{\log z} - \sin \frac{z_n}{\log z_n} = (z-z_n) \cdot \left(\frac{d}{dt} \left( t \mapsto \sin \frac{t}{\log t}\right) \right)\Big\vert_{t=\xi},
$$ 
for some $\xi\in [z_n, z]$. Therefore,
\begin{align*}
 \sin \frac{z}{\log z} - \sin \frac{z_n}{\log z_n} =O(n) \cdot O(1) \cdot \frac{\log \xi-1}{\log^2 \xi} = O(n)\cdot O(\frac1{n^2}) = O(\frac1{n}).
\end{align*}

Hence,
\begin{align*}
J&= \int_{z_n}^{z_n+\frac{n}d \log 2} \left(\sin \frac{z_n}{\log z_n} + O(\frac1{n}) \right)dz +O(1) \\
&= \int_{z_n}^{z_n+\frac{n}d \log 2} dz +O(1) = n \frac{\log 2}d  +O(1). 
\end{align*}

Consequently, 
$$
 \lim_{n\to\infty} \sup_{m\ge0} \frac1n\sum_{k=m}^{m+n-1} \int_{\frac{k}d \log2}^{\frac{k+1}d \log 2} \sin \frac{z}{\log z} dz \ge \frac{\log 2}d ,$$
 which proves the assertion. 
\end{proof}

The following theorem provides the example of a Dixmier measurable pseudo-differential operator 
such that the class of all positive normalised traces does not coincide on this operator.
 
\begin{thm} \label{psdo}
 There exists a compactly supported pseudo-differential operator $Q$ of order $-d$ such that
 $Q$ is Dixmier-measurable but $Q$ is not $\mathcal{PT}$-measurable.
\end{thm}

\begin{proof}
 The construction of the operator $Q$ is similar, at least in spirit, to that of \cite[Proposition 6.19]{KLPS} (see also~\cite[Proposition 11.3.22]{LSZ}).
Consider the function
$$q(s):= |s|^{-d} \sin \left(\frac{\log |s|}{\log \log |s|}\right), \ s \in \mathbb R^d, \ |s| \ge 4.$$

Similarly to~\cite[Proposition 10.2.10]{LSZ} it can be proved that $q$ is a symbol of some pseudo-differential operator, say $Q'$.
Let $\phi \in C_c^\infty (\mathbb R^d)$ be such that $\|\phi\|_2= ({\rm Vol} \ \mathbb S^{d-1})^{-1/2}$, 
where $\mathbb S^{d-1}:= \{ s \in \mathbb R^d: |s|=1\}$ is $(d-1)$-sphere. 
The operator $Q=M_\phi Q' M_\phi^*$ is compactly supported. 
By~\cite[Lemma 6.18]{KLPS} the principal symbol of $Q$ is
$$(x,s) \mapsto  |\phi(x)|^2 q(s), \ x,s \in \mathbb R^d.$$

To show that $Q$ is Laplacian modulated we check the condition~\eqref{mod}.
We have for $t>4$
$$\int_{|s|>t} \int_{\mathbb R^d} |p_A(x,s)|^2 \ dx ds 
 = \int_{\mathbb R^d} |\phi(x)|^2 dx \int_{|s|> t} q^2(s) ds.$$
The transformation from spherical coordinates to Cartesian gives
\begin{align*}
 \int_{|s|>t} \int_{\mathbb R^d} |p_A(x,s)|^2 \ dx ds 
 = \int_t^\infty r^{-2d} \sin^2 \left(\frac{\log r}{\log \log r}\right) r^{d-1} \ dr
 = O(t^{-d}).
\end{align*}
Hence,
$$\sup_{t>4} (1+t)^{d/2} \left( \int_{|s|>t} \int{\mathbb R^d} |p_A(x,s)|^2 \ dx ds\right)^{1/2}<\infty$$
and by~\eqref{mod} the operator $Q$ is Laplacian modulated.

For every $n \ge 4^d$ we have
$$\int_{\mathbb R^d}\int_{|s|\le n^{1/d}} |\phi(x)|^2 q(s) ds\,dx 
= \int_{\mathbb R^d} |\phi(x)|^2 dx \int_{4 \leq |s|\le n^{1/d}} q(s) ds + O(1).$$
The transformation from spherical coordinates to Cartesian gives
\begin{equation}\label{e12}
\begin{aligned}
\int_{\mathbb R^d}\int_{|s|\le n^{1/d}} |\phi(x)|^2 q(s) ds\,dx
&=  \int_4^{n^{1/d}} \sin \left(\frac{\log r}{\log \log r}\right)r^{-d}r^{d-1}dr + O(1) \\
&= \int_4^{n^{1/d}}\sin \left(\frac{\log r}{\log \log r}\right) \frac{dr}r + O(1)\\
&= \int_{\log4}^{\log n^{1/d}} \sin \left(\frac{z}{\log z}\right)dz\\
&=O(\log \log n) =o(\log n),
\end{aligned} 
\end{equation}
where the penultimate equality is provided by Lemma~\ref{lemma1}.


Combining this observation with Definition~\ref{def:res}, we obtain
\begin{align*}
{\rm Res}(Q)&= \left[ \frac1{\log(2+n)} \int_{\mathbb R^d} \int_{|s|\le n^{1/d}} p_Q(x,s) ds\,dx \right] \\
&=\left[ \frac1{\log(2+n)}\int_4^{n^{1/d}}\sin \left(\frac{\log r}{\log \log r}\right) \frac{dr}r \right]=0.
\end{align*}

So, ${\rm Res}(Q)$ is a scalar and, by Theorem~\ref{CTTv1}(ii), ${\rm Tr}_\omega(Q)=0$ for every Dixmier trace ${\rm Tr}_\omega$.

Combining the result of Theorem~\ref{CTT} with~\eqref{e12}, we conclude that 
the pseudo-differential operator $Q$ is $\mathcal{PT}$-measurable if and only if the sequence
\begin{equation}
 \left\{ \int_{\log2^{n/d}}^{\log 2^{(n+1)/d}} \sin \frac{z}{\log z} dz \right\}_{n\ge0}
\end{equation}
is almost convergent (to zero, since all Dixmier traces vanish on $Q$).
However, this is not the case due to Lemma~\ref{lemma2}.

\end{proof} 

\begin{rem}\label{last_rem}
The example of the operator $Q$ given in Theorem~\ref{psdo} is interesting, 
because it shows how different traces of compactly supported pseudo-differential operators of order $-d$ 
are from traces of classical compactly supported pseudo-differential operators of order $-d$. 
On the classical operators there is one trace and one Wodzicki's residue. 
Even the natural vector generalisation ${\rm Res}$ of the Wodzicki's residue, 
whilst it does capture the behaviour of Dixmier traces on the non-classical operators, 
still does not capture the full behaviour of positive traces on the non-classical operators. 
Indeed, the proof of Theorem~\ref{psdo} shows that ${\rm Res}(Q)$ 
may be zero but still there are positive normalised traces which yield a non-zero value on $Q$. 
\end{rem}

\textbf{Acknowledgements:} 
 We thank S.~Lord and G.~Levitina for the detailed reading and useful discussions of the manuscript and suggesting a number of improvements.


\end{document}